\documentclass{scrartcl}
\usepackage[T1]{fontenc}
\usepackage[utf8]{inputenc}
\usepackage{amsmath, amsthm, amssymb}
\usepackage{dsfont}
\usepackage{graphicx}
\usepackage[format=plain]{caption}
\usepackage{capt-of}
\usepackage{subcaption}
\usepackage{bm}
\usepackage{xcolor}

\newcommand{\norm}[1]{\left\| #1 \right\|}  
\newcommand{\scprd}[1]{\left\langle #1 \right\rangle}  
\renewcommand{\d}{\,\mathrm{d}} 
\newcommand{\e}{\mathrm{e}} 
\newcommand{\ddt}[1]{\frac{\mathrm{d}} {\mathrm{d}t} #1}   
\newcommand{\N}{\mathbb{N}}  

\newcommand{\R}{\mathbb{R}}

\newcommand{\dist}{\operatorname{dist}}

\DeclareMathOperator*{\esssup}{ess-sup}  
\newcommand{\eps}{\varepsilon}
\renewcommand{\phi}{\varphi}
\newcommand{\ul}{\underline}
\newcommand{\ol}{\overline}

\numberwithin{equation}{section}

\newtheorem{thm}{Theorem}[section]
\newtheorem{cor}[thm]{Corollary}

\newtheorem{prop}[thm]{Proposition}
\newtheorem{lm}[thm]{Lemma}
\newtheorem{cond}[thm]{Condition}

\theoremstyle{definition} 
\theoremstyle{definition} 
\theoremstyle{definition}

\title{Asymptotic gain results for attractors of semilinear systems}

\author{Jochen Schmid$^{1,2}$, Oleksiy Kapustyan$^{3}$, Sergey Dashkovskiy$^{1}$ \\  
\small $^1$ Institute for Mathematics, University of W\"urzburg, 97074 W\"urzburg, Germany\\
\small $^2$ Fraunhofer Institute for Industrial Mathematics (ITWM), 67663 Kaiserslautern, Germany\\ 
\small $^3$ Kiev National Taras Shevchenko University, 01033 Kiev, Ukraine\\
\small jochen.schmid@itwm.fraunhofer.de, alexkap@univ.kiev.ua,\\ \small sergey.dashkovskiy@mathematik.uni-wuerzburg.de}  

\date{}

\begin{document}

\maketitle

\begin{abstract}
\small{ \noindent 
We establish asymptotic gain along with input-to-state practical stability results 
for disturbed semilinear systems w.r.t.~the global attractor of the respective undisturbed system. We apply our results to a large class of nonlinear reaction-diffusion equations comprising disturbed Chaffee--Infante equations, for example.    
}
\end{abstract}


\section{Introduction}

In this paper, we are concerned with disturbed infinite-dimensional semilinear systems of the form
\begin{align} \label{eq:semilin-evol-eq-u,intro}
\dot{x}(t) = Ax(t) + g(x(t)) + h u(t).
\end{align}
In this equation, $A: D(A) \subset X \to X$ is a linear semigroup generator on $X$, $g$ a nonlinear function in $X$, $h: U \to X$ is a bounded linear operator from $U$ to $X$, $X$ and $U$ are Banach spaces, and $u: [0,\infty) \to U$ is an external disturbance signal.  
In particular, we consider disturbed nonlinear reaction-diffusion equations of the form
\begin{align} \label{eq:react-diffus-eq,intro}
\begin{split}
\partial_t y(t,\zeta) &= \Delta y(t,\zeta) + g(y(t,\zeta)) + h(\zeta) u(t) \qquad (\zeta \in \Omega) \\
y(t,\zeta) &= 0 \qquad (\zeta \in \partial \Omega)
\end{split} 
\end{align}
on a bounded domain $\Omega \subset \R^d$, where \textcolor{black}{$d \in \N$}, \textcolor{black}{$\Delta$ is the Dirichlet Laplacian and} $h \in X := L^2(\Omega,U)$ and $U := \R$.
We are interested in the asymptotic behavior of the solutions to~\eqref{eq:semilin-evol-eq-u,intro} and~\eqref{eq:react-diffus-eq,intro} w.r.t.~a global attractor of the respective undisturbed system, that is, system~\eqref{eq:semilin-evol-eq-u,intro} or~\eqref{eq:react-diffus-eq,intro} with $u := 0$. So, we assume that the undisturbed system has a global attractor $\Theta$, 
that is, a compact invariant and uniformly attractive subset of $X$. 
And then we show that -- under suitable dissipativity and compactness assumptions -- the 
disturbed system~\eqref{eq:semilin-evol-eq-u,intro} or~\eqref{eq:react-diffus-eq,intro} is of asymptotic gain and input-to-state practically stable w.r.t.~$\Theta$. 
Spelled out, the asymptotic gain property means that there exists a comparison function $\gamma \in \mathcal{K}$ such that
\begin{align} \label{eq:AG-def,intro}
\limsup_{t\to\infty} \norm{x(t,x_0,u)}_{\Theta} \le \gamma(\norm{u}_{\infty})
\end{align}
for every $(x_0,u) \in X \times \mathcal{U}$, and input-to-state practical stability, in turn, means that there exist comparison functions $\beta \in \mathcal{KL}$, $\gamma \in \mathcal{K}$ and a constant $c \in [0,\infty)$ such that
\begin{align} \label{eq:ISpS-def,intro}
\norm{x(t,x_0,u)}_{\Theta} \le \beta(\norm{x_0}_{\Theta},t) + \gamma(\norm{u}_{\infty}) + c \qquad (t \in [0,\infty))
\end{align}
for every $(x_0,u) \in X \times \mathcal{U}$. If $c = 0$ in~\eqref{eq:ISpS-def,intro}, one speaks of input-to-state stability. In the relations~\eqref{eq:AG-def,intro} and~\eqref{eq:ISpS-def,intro} above, $x(\cdot,x_0,u)$ denotes the solution of~\eqref{eq:semilin-evol-eq-u,intro} with initial value $x_0$ -- the solution being understood in some appropriate sense. 
In the special case~\eqref{eq:react-diffus-eq,intro}, this will be the mild sense or the weak sense. 
Also, the set $\mathcal{U}$ of admissible disturbances is a suitable subset of $L^{\infty}([0,\infty),U)$,
\textcolor{black}{
the distance from a point to a set is defined by}
\begin{align} \label{eq:norm-Theta,def}
\norm{x}_{\Theta} := \dist(x,\Theta) := \inf_{\theta \in \Theta} \norm{x-\theta} \qquad (x \in X),
\end{align}
\textcolor{black}{while for any two sets $A,B\subset X$ the Hausdorff semi-distance $\dist(A,B)$ is defined by
\begin{align*}
\dist(A,B):=\sup\limits_{x\in A} \dist(x,B) = \sup\limits_{x\in A}\inf\limits_{y\in B}\|x-y\|.
\end{align*}
And finally,} $\mathcal{KL}$, $\mathcal{K}$ denote the standard comparison function classes recalled in~\eqref{eq:comparison-fct-classes} below. 
\smallskip

In order to prove our asymptotic gain result, we embed the equation~\eqref{eq:semilin-evol-eq-u,intro} 
for every $u \in \mathcal{U}$ into a whole family of equations
\begin{align} \label{eq:semilin-evol-eq-v,intro}
\dot{x}(t) = Ax(t) + g(x(t)) + h v(t) \qquad (v \in \mathcal{V}(u))
\end{align}
parametrized by $v \in \mathcal{V}(u)$, where $\mathcal{V}(u)$ is a suitably chosen subset of $L^2_{\mathrm{loc}}([0,\infty),U)$ satisfying $u \in \mathcal{V}(u)$ and $\mathcal{V}(0) = \{0\}$ \textcolor{black}{(see~\eqref{eq:V(u)-def} below)}.
We then consider the dynamical map -- or, more precisely, the set-valued semiprocess -- $S_{\mathcal{V}(u)}$ corresponding to the family of equations~\eqref{eq:semilin-evol-eq-v,intro} and show 
that $S_{\mathcal{V}(u)}$ for every $u \in \mathcal{U}$ has a unique global attractor $\Theta_{\mathcal{V}(u)}$ which depends upper semicontinuously on $u$ in the sense that 
\begin{align} \label{eq:upper-semicontinuity,intro}
\dist( \Theta_{\mathcal{V}(u)}, \Theta ) := \sup_{\theta_u \in \Theta_{\mathcal{V}(u)}} \dist(\theta_u, \Theta) \longrightarrow 0 \qquad (u \to 0).
\end{align}
With the help of this upper semicontinuity \textcolor{black}{(Definition~1.4.1 of~\cite{AubinFrankowska} and the remark after it)}, we can then conclude our asymptotic gain result in a relatively simple way. 
\smallskip

\textcolor{black}{As is known from the literature~\cite{Robinson}, \cite{Temam}, the global attractor is one of the most interesting objects in the state space of a dissipative system. }
As far as we know, our results are essentially the first asymptotic gain and input-to-state practical stability results w.r.t.~attractors of infinite-dimensional systems like~\eqref{eq:semilin-evol-eq-u,intro} and, especially, of concrete partial differential equation systems like~\eqref{eq:react-diffus-eq,intro}. We are aware of one other input-to-state practical stability result w.r.t.~to attractors of semilinear systems~\cite{Mi18-ISpS}, but that result is a theoretical characterization of input-to-state practical stability in terms of rather abstract conditions which are probably not easily verifiable in practice. 
All input-to-state stability results for concrete \textcolor{black}{pde} systems -- like those from~\cite{DaMi13}, \cite{JaNaPaSc16}, \cite{JaSc18}, \cite{JSZ18}, \cite{KaKr16}, \cite{KaKr17}, \cite{MaPr11}, \cite{MiKaKr17}, \cite{MiP19}, \cite{ScZw18}, \cite{TaPrTa17}, \cite{Schwe19}, \cite{ZhZh17a}, \cite{ZhZh17b},  \textcolor{black}{\cite{ZhZh19}} -- however, establish input-to-state stability only w.r.t.~an equilibrium point $\theta$ of the respective undisturbed system and, without loss of generality, this equilibrium point is then assumed to be $\theta = 0$. 
In particular, the results from those papers do not cover the Chaffee--Infante equation \textcolor{black}{with a non-singleton attractor}, for example, that is, the reaction-diffusion equation~\eqref{eq:react-diffus-eq,intro} with nonlinearity $g$ given by
\begin{align*}
g(r) := \lambda(-r^3 + r) \qquad (r \in \R)
\end{align*} 
\textcolor{black}{with some $\lambda \in (0,\infty)$. Indeed, in general} the respective undisturbed system only has a non-singleton attractor $\Theta \supsetneq \{\theta\} = \{0\}$.  \textcolor{black}{See Section~11.5 of~\cite{Robinson} where it is shown that for $\Omega = (0,\pi)$ and $\lambda \in (n^2,(n+1)^2)$ with some $n \in \N$, the corresponding undisturbed Caffee--Infante equation has exactly $2n+1$ equilibrium points. In particular, it has a non-singleton global attractor whenever $\lambda > 1$.} With our results, by contrast, we can cover the Chaffee--Infante equation \textcolor{black}{with a non-singleton attractor} and 
many more nonlinearities. We refer to~\cite{KaKaVa15}, \cite{GoKaKaPa14}, \cite{GoKakaKh15} \cite{DaKaRo17} for other interesting results about non-trivial global attractors of nonlinear, impulsive, or even multi-valued semigroups. 
\smallskip

Also, our strategy of proving the asymptotic gain property~\eqref{eq:AG-def,intro} -- by embedding the original system~\eqref{eq:semilin-evol-eq-u,intro} into the family of systems~\eqref{eq:semilin-evol-eq-v,intro} and by then establishing the upper semicontinuity~\eqref{eq:upper-semicontinuity,intro} -- seems to be new, too. 
At least, this strategy of proof is completely different from more traditional approaches -- like input-to-state Lyapunov function approaches, for instance, which might come to mind first, in view of corresponding  finite-dimensional results~\cite{LiSoWa96}, \cite{SoWa96}. 
\smallskip

In the entire paper, we will use the following conventions and notations. We will write $\R^+_0 := [0,\infty)$, $X$ and $U$ will be Banach spaces over the reals $\R$, the norm of $X$ will be denoted simply by $\norm{\cdot}$ and similarly the scalar product of $X$, when existent, will be denoted by $\scprd{\cdot,\cdot\cdot}$. Also, for $u \in L^p_{\mathrm{loc}}(\R^+_0,U)$, $p \in [1,\infty) \cup \{\infty\}$, \textcolor{black}{$t_0>0$} we will write 
\begin{align*}
\norm{u}_{[0,t_0],p} := \norm{u|_{[0,t_0]}}_p 
\color{black}
:=
\begin{cases}
\big( \int_0^{t_0} |u(s)|^p \d s \big)^{1/p}, \qquad p < \infty, \\
\esssup_{s \in [0,t_0]} |u(s)|, \qquad p = \infty,
\end{cases} 
\color{black}
\end{align*}
for short. 
As usual, 
\begin{align*}
B_r(x_0) = B_r^{X}(x_0), \quad \ol{B}_r(x_0) = \ol{B}_r^{X}(x_0) \quad \text{and} \quad B_r(u_0) = B_r^{\mathcal{U}}(u_0), \quad  \ol{B}_r(u_0) = \ol{B}_r^{\mathcal{U}}(u_0)
\end{align*}
denote the open and closed balls in $X$ or $\mathcal{U}$ of radius $r$ around $x_0 \in X$ or $u_0 \in \mathcal{U}$ respectively. We will often use the notation~\eqref{eq:norm-Theta,def} and
\begin{align*}
B_r(\Theta) := \{x \in X: \norm{x}_{\Theta} < r \} 
\qquad \text{and} \qquad
\ol{B}_r(\Theta) := \{x \in X: \norm{x}_{\Theta} \le r \},
\end{align*}
as well as the notation 
$\dist(M,\Theta) := \sup_{x\in M} \norm{x}_{\Theta}$
for subsets $M, \Theta \subset X$. Also, $\mathcal{K}$, $\mathcal{K}_{\infty}$ and $\mathcal{KL}$ will denote the following standard classes of comparison functions:
\begin{gather}
\mathcal{K} := \{ \gamma \in C(\R^+_0,\R^+_0): \gamma \text{ strictly increasing with } \gamma(0) = 0 \} \notag \\
\mathcal{K}_{\infty} := \{ \gamma \in \mathcal{K}: \gamma \text{ unbounded} \}  \label{eq:comparison-fct-classes}\\
\mathcal{KL} := \{ \beta \in C(\R^+_0 \times \R^+_0,\R^+_0): \beta(\cdot,t) \in \mathcal{K} \text{ for } t \ge 0 \text{ and } \beta(s,\cdot) \in \mathcal{L} \text{ for } s > 0 \}, \notag
\end{gather}
where $\mathcal{L} := \{ \gamma \in C(\R^+_0,\R^+_0): \gamma \text{ strictly decreasing with } \lim_{t\to\infty} \gamma(t) = 0 \}$. And finally, when speaking merely of a semigroup -- as opposed to a nonlinear semigroup~\cite{Robinson}, \cite{Temam} -- we will mean a strongly continuous semigroup of linear operators~\cite{EnNa}, \cite{Pazy}.

\section{Some preliminaries}

\subsection{Semiprocesses} 

We begin by defining the central dynamical objects of this paper, namely semiprocess families. 
Suppose $\mathcal{V}$ is a 
topological space and $T(h)$ are maps leaving $\mathcal{V}$ invariant:
\begin{align} \label{eq:V-invariant-under-T}
T(h)\mathcal{V} \subset \mathcal{V} \qquad (h \in \R^+_0).
\end{align}
A family $(S_v)_{v\in \mathcal{V}}$ of maps $S_v: \Delta \times X \to X$ where $\Delta:= \{(t,s) \in \R^+_0 \times \R^+_0: t \ge s \}$ is called a \emph{semiprocess family (on $X$ with input space $\mathcal{V}$ and translation $T$)} iff the following conditions are satisfied:
\begin{itemize}
\item[(i)] $S_v$ is a semiprocess for every $v \in \mathcal{V}$, that is, 
\begin{align*}
S_v(s,s,x) = x \qquad \text{and} \qquad S_v\big(t,s,S_v(s,r,x)\big) = S_v(t,r,x)
\end{align*}
for all $t \ge s \ge r$, $x \in X$ and $v \in \mathcal{V}$
\item[(ii)] $S_v(t+h,s+h,x) = S_{T(h)v}(t,s,x)$ for all $t \ge s$, $h \ge 0$, $x \in X$ and $v \in \mathcal{V}$.
\end{itemize}


So, a semiprocess family simply consists of dynamical maps $S_v$ for every input $v \in \mathcal{V}$ and these dynamical maps are connected via the translation $T$. In all our results below, the input space $\mathcal{V}$ will be a subset of $L^2_{\mathrm{loc}}(\R^+_0,U)$ endowed with a suitable 
topology and the translation $T$ will be the left-translation semigroup on $L^2_{\mathrm{loc}}(\R^+_0,U)$, that is,
\begin{align} \label{eq:transl-sgr}
T(h)v = v(\cdot+h) \qquad (v \in L^2_{\mathrm{loc}}(\R^+_0,U), h \in \R^+_0). 
\end{align}
In our applications, we will often use the following alternative notations
\begin{align} \label{eq:alternative-notation-for-trajectories}
x(t,s,x_s,v) := S_v(t,s,x_s) \qquad \text{and} \qquad x(t,x_0,v) := S_v(t,0,x_0),
\end{align}
\textcolor{black}{where $x_s$ is the initial state of the Cauchy problem of~\eqref{eq:semilin-evol-eq-v,intro} at initial time $s$ (see Lemma~\ref{lm:max-mild-sol} below)}. 
It should be noticed that every semiprocess family satisfies the following so-called cocycle property:
\begin{align} \label{eq:cocycle-prop}
S_v(t+h,0,x) = S_v\big(t+h,h,S_v(h,0,x)\big) = S_{T(h)v}\big(t,0,S_v(h,0,x)\big)
\end{align}
for all $t, h \ge 0$, $x \in X$ and $v \in \mathcal{V}$.
It should also be noticed that semiprocess families are closely related to the  (forward-complete) dynamical systems with inputs from~\cite{MiWi16a}, \cite{Sc18-wISS}. Indeed, every semiprocess family $(S_u)_{u \in \mathcal{U}}$ with input space $\mathcal{U} \subset L^{\infty}(\R^+_0,U)$ and translation $T$ given by~\eqref{eq:transl-sgr} determines a dynamical system $(X,\mathcal{U},\phi)$ 
with inputs in the sense of~\cite{MiWi16a}, \cite{Sc18-wISS}
via
\begin{align} \label{eq:from-semiproc-family-to-dyn-syst-with-inputs}
\phi(t,x,u) := S_u(t,0,x) \qquad (t \in \R^+_0 \text{ and } (x,u) \in X \times \mathcal{U}),
\end{align}
provided that the trajectories $t \mapsto S_u(t,0,x)$ are continuous and that $\mathcal{U}$ is invariant under concatenations. And conversely, every dynamical system $(X,\mathcal{U},\phi)$ 
with inputs in the sense of~\cite{MiWi16a}, \cite{Sc18-wISS} and with $\mathcal{U} \subset L^{\infty}(\R^+_0,U)$ determines a semiprocess familiy $(S_u)_{u\in\mathcal{U}}$ with translation $T$ given by~\eqref{eq:transl-sgr} via
\begin{align} \label{eq:from-dyn-syst-with-inputs-to-semiproc-family}
S_u(t,s,x) := \phi(t-s,x,u(\cdot+s)) \qquad ((t,s) \in \Delta \text{ and } (x,u) \in X \times \mathcal{U}).
\end{align}

Whenever a semiprocess family $(S_v)_{v\in\mathcal{V}}$ is given, we will denote by $S_{\mathcal{V}}$ the corresponding \emph{set-valued semiprocess} which is defined by
\begin{align} \label{eq:set-valued-semiproc,def}
S_{\mathcal{V}}(t,s,x) := \big\{ S_v(t,s,x): v \in \mathcal{V} \big\}.
\end{align}
Similarly, for every subset $M \subset X$ we will write 
\begin{align*}
S_{\mathcal{V}}(t,s,M) := \bigcup_{v \in \mathcal{V}}  S_v(t,s,M) = \big\{ S_v(t,s,x): x \in M \text{ and } v \in \mathcal{V} \big\}.
\end{align*}
\textcolor{black}{With these notations, we obtain the following semigroup-like property~\cite{Robinson}, \cite{Temam} from the cocycle property~\eqref{eq:cocycle-prop} and the invariance property~\eqref{eq:V-invariant-under-T}:
\begin{align} \label{eq:sgr-property-S_V}
S_{\mathcal{V}}(t+h,0,M) \subset S_{T(h)\mathcal{V}}\big(t,0,S_{\mathcal{V}}(h,0,M)\big) \subset S_{\mathcal{V}}\big(t,0,S_{\mathcal{V}}(h,0,M)\big)
\end{align}
for all $t,h \ge 0$. We have to work with set-valued semiprocesses because for them (as opposed to the individual semiprocesses $S_v$ they consist of), a nice attractor theory analogous to the attractor theory for semigroups~\cite{Robinson}, \cite{Temam} can be developed. See the beginning of Section~3 for a more detailed explanation why we need to work with set-valued semiprocesses.}

\subsection{Attractors} 

We now move on to define attractors of set-valued semiprocesses. Suppose $(S_v)_{v\in \mathcal{V}}$ is a semiprocess family on $X$. A subset $\Theta_{\mathcal{V}} \subset X$ is then called a \emph{global attractor for the set-valued semiprocess $S_{\mathcal{V}}$} iff $\Theta_{\mathcal{V}}$ is compact and the following conditions are satisfied:
\begin{itemize}
\item[(i)] $\Theta_{\mathcal{V}}$ is \emph{uniformly attractive for $S_{\mathcal{V}}$}, that is, for every bounded subset $B \subset X$ one has
\begin{align} \label{eq:def-uniformly-attractive}
\dist\big( S_{\mathcal{V}}(t,0,B), \Theta_{\mathcal{V}} \big) \longrightarrow 0 \qquad (t \to \infty)
\end{align}
\item[(ii)] $\Theta_{\mathcal{V}}$ is \emph{negatively invariant under $S_{\mathcal{V}}$}, that is, 
\begin{align} \label{eq:def-neg-invar}
\Theta_{\mathcal{V}} \subset S_{\mathcal{V}}(t,0,\Theta_{\mathcal{V}}) \qquad (t \in \R^+_0). 
\end{align}
\end{itemize} 

\textcolor{black}{See~\cite{KaVa09} (Definition~21), for instance.} Also see the remark at the very end of Section~\ref{sect:wAG+ISpS,def} for the relation of global attractors for set-valued semiprocesses and of global attractors for nonlinear semigroups~\cite{Robinson}, \cite{Temam}. 
It immediately follows from the definition above 
that a global attractor of a set-valued semiprocess $S_{\mathcal{V}}$ is contained in every closed uniformly $S_{\mathcal{V}}$-attractive set. 
And from this, in turn, it is clear that a set-valued semiprocess $S_{\mathcal{V}}$ can have at most one global attractor. 
\smallskip

We therefore turn to the question of existence of attractors now. In this context, the following elementary lemma is useful.

\begin{lm} \label{lm:char-omega-limit-sets}
If $(S_v)_{v\in \mathcal{V}}$ is \textcolor{black}{any} semiprocess family on $X$, then for \textcolor{black}{any subset} $M \subset X$ its $\omega$-limit set $\omega_{\mathcal{V}}(M) := \bigcap_{\tau \ge 0} \ol{ \bigcup_{t \ge \tau} S_{\mathcal{V}}(t,0,M) }$ under $S_{\mathcal{V}}$ can be characterized as
\begin{align*}
\omega_{\mathcal{V}}(M) = \Big\{ \xi \in X: \xi = \lim_{n\to\infty} S_{v_n}(t_n,0,x_n) \text{ where } t_n \longrightarrow \infty \text{ and } (x_n,v_n) \in M \times \mathcal{V} \Big\}.
%
\end{align*}
\end{lm}

\begin{proof}
See, for instance, Lemma~3.3 from~\cite{KaMeVa03}. 
\end{proof}

Additionally, 
the following notions for a set-valued semiprocess $S_{\mathcal{V}}$ corresponding to a semiprocess family $(S_v)_{v\in \mathcal{V}}$ on $X$ are useful. A subset $B_0 \subset X$ is called \emph{absorbing for $S_{\mathcal{V}}$} iff for every bounded set $B \subset X$ there exists a time $\tau_B$ such that
\begin{align*} 
S_{\mathcal{V}}(t,0,B) \subset B_0 \qquad (t \ge \tau_B).
\end{align*}
$S_{\mathcal{V}}$ is called \emph{asymptotically compact} iff for every sequence $(t_n,x_n,v_n)$ in $\R^+_0 \times X \times \mathcal{V}$ with
\begin{align*} 
t_n \longrightarrow \infty 
\qquad \text{and} \qquad
\sup_{n\in \N} \norm{x_n} < \infty,
\end{align*}
the sequence $\big(S_{v_n}(t_n,0,x_n)\big)$ has a convergent subsequence. See the remark after \textcolor{black}{the proof of} Lemma~\ref{lm:existence-attractor} below for a set-theoretic characterization of asymptotic compactness. 
With these notions at hand, we can now formulate a basic existence result for attractors of set-valued semiprocesses. 

\begin{lm} \label{lm:existence-attractor}
Suppose $(S_v)_{v\in \mathcal{V}}$ is a semiprocess family on $X$ such that the input space $\mathcal{V}$ is compact and first-countable and the following assumptions are satisfied:
\begin{itemize}
\item[(i)] $S_{\mathcal{V}}$ has a bounded absorbing set $B_0$ 
\item[(ii)] $S_{\mathcal{V}}$ is asymptotically compact
\item[(iii)] $X \times \mathcal{V} \ni (x,v) \mapsto S_v(t_0,0,x)$ is continuous for every $t_0 \in (0,\infty)$. 
\end{itemize}
Then $S_{\mathcal{V}}$ has a unique global attractor $\Theta_{\mathcal{V}}$, namely $\Theta_{\mathcal{V}} = \omega_{\mathcal{V}}(B_0) = \bigcap_{\tau \ge 0} \ol{ \bigcup_{t \ge \tau} S_{\mathcal{V}}(t,0,B_0) }$.
\end{lm}

\begin{proof}
We can proceed similarly to the proof of Theorem~22 from~\cite{KaVa09}, but for the reader's convenience we give a self-contained 
proof here. 
As a first step, we show that $\omega_{\mathcal{V}}(B)$ is compact for every bounded subset $B \subset X$. 
So, let $B \subset X$ be bounded and let $(\xi_n)$ be an arbitrary sequence in $\omega_{\mathcal{V}}(B)$. In view of Lemma~\ref{lm:char-omega-limit-sets} there exists for every $n \in \N$ some $(t_n,x_n,v_n) \in \R^+_0 \times X \times \mathcal{V}$ with $t_n \ge n$ and $x_n \in B$ such that
\begin{align} \label{eq:existence-attractor-step-1}
\norm{S_{v_n}(t_n,0,x_n) - \xi_n} \le 1/n.
\end{align}
Since $S_{\mathcal{V}}$ is asymptotically compact by assumption~(ii), 
$(S_{v_n}(t_n,0,x_n))$ has a convergent subsequence $(S_{v_{n_k}}(t_{n_k},0,x_{n_k}))$ and, by Lemma~\ref{lm:char-omega-limit-sets}, its limit $\xi$ belongs to $\omega_{\mathcal{V}}(B)$. So, by virtue of~\eqref{eq:existence-attractor-step-1}, 
\begin{align}
\xi_{n_k} \longrightarrow \xi \in \omega_{\mathcal{V}}(B) \qquad (k \to \infty),
\end{align}
which proves the claimed compactness of $\omega_{\mathcal{V}}(B)$.
\smallskip

As a second step, we show that $\dist(S_{\mathcal{V}}(t,0,B),\omega_{\mathcal{V}}(B)) \longrightarrow 0$ as $t \to \infty$ for every bounded subset $B \subset X$. 
Assuming the contrary, we find a bounded subset $B \subset X$ and an $\eps > 0$ such that for every $n \in \N$ there exists a $t_n \ge n$ with
\begin{align*}
S_{\mathcal{V}}(t_n,0,B) \not\subset B_{\eps}( \omega_{\mathcal{V}}(B) ),
\end{align*}
that is, there exist $x_n \in B$ and $v_n \in \mathcal{V}$ such that
\begin{align} \label{eq:existence-attractor-step-2,1}
S_{v_n}(t_n,0,x_n) \notin  B_{\eps}( \omega_{\mathcal{V}}(B) ) \qquad (n \in \N). 
\end{align}
Since $S_{\mathcal{V}}$ is asymptotically compact by assumption~(ii), 
$(S_{v_n}(t_n,0,x_n))$ has a convergent subsequence $(S_{v_{n_k}}(t_{n_k},0,x_{n_k}))$ and, by Lemma~\ref{lm:char-omega-limit-sets}, its limit belongs to $\omega_{\mathcal{V}}(B)$, that is,
\begin{align} \label{eq:existence-attractor-step-2,2}
\lim_{k\to\infty} S_{v_{n_k}}(t_{n_k},0,x_{n_k}) \in \omega_{\mathcal{V}}(B).
\end{align}
Contradiction to~\eqref{eq:existence-attractor-step-2,1}! So, our 
assumption was false and the second step is proved.
\smallskip

As a third step, we show that $\Theta_{\mathcal{V}} := \omega_{\mathcal{V}}(B_0)$ is compact and uniformly attractive for $S_{\mathcal{V}}$. 
In view of the first step, the compactness of $\Theta_{\mathcal{V}}$ is clear and we only have to prove the uniform attractivity of $\Theta_{\mathcal{V}}$. So, let $B \subset X$ be an arbitrary bounded subset. Since $B_0$ is absorbing for $S_{\mathcal{V}}$ by assumption~(i), there exists a time $\tau_B$ such that
\begin{align}
S_{\mathcal{V}}(t,0,B) \subset B_0 \qquad (t \ge \tau_B).
\end{align}
Also, by virtue of the second step, 
for every $\eps > 0$ there exists a time $\tau_{B_0,\eps}$ such that
\begin{align}
S_{\mathcal{V}}(t,0,B_0) \subset B_{\eps}(\Theta_{\mathcal{V}}) \qquad (t \ge \tau_{B_0,\eps}).
\end{align}  
So, for all $t \ge \tau_B + \tau_{B_0,\eps}$, we conclude with~\eqref{eq:sgr-property-S_V} that 
\begin{align}
S_{\mathcal{V}}(t,0,B) \subset S_{T(\tau_B)\mathcal{V}}\big(t-\tau_B,0,S_{\mathcal{V}}(\tau_B,0,B)\big) \subset S_{\mathcal{V}}(t-\tau_B,0,B_0) \subset B_{\eps}(\Theta_{\mathcal{V}}).
\end{align}
And therefore, $\Theta_{\mathcal{V}}$ is uniformly attractive for $S_{\mathcal{V}}$, as desired.
\smallskip

As a fourth and last step, we show that $\Theta_{\mathcal{V}}$ is negatively invariant under $S_{\mathcal{V}}$. 
So, let $\xi \in \Theta_{\mathcal{V}} = \omega_{\mathcal{V}}(B_0)$ and $t \in (0,\infty)$. In view of Lemma~\ref{lm:char-omega-limit-sets}, $\xi = \lim_{n\to\infty} S_{v_n}(t_n,0,x_n)$ for certain $(t_n,x_n,v_n) \in \R^+_0\times X \times \mathcal{V}$ with $t_n \longrightarrow \infty$ and $x_n \in B_0$. Also, by the cocycle property, 
\begin{align} \label{eq:existence-attractor-step-4,1}
S_{v_n}(t_n,0,x_n) = S_{T(t_n-t)v_n}\big(t,0,S_{v_n}(t_n-t,0,x_n)\big)
\end{align}
for all $n \in \N$ so large that $t_n \ge t$. Since $\mathcal{V}$ is compact and first-countable and hence sequentially compact and since $S_{\mathcal{V}}$ is asymptotically compact by assumption~(i), 
there exist subsequences such that
\begin{align} \label{eq:existence-attractor-step-4,2}
T(t_{n_k}-t)v_{n_k} \longrightarrow v \qquad \text{and} \qquad S_{v_{n_k}}(t_{n_k}-t,0,x_{n_k}) \longrightarrow x
\end{align}
for some $v \in \mathcal{V}$ and some $x \in X$, which actually belongs to $\omega_{\mathcal{V}}(B_0) = \Theta_{\mathcal{V}}$ by virtue of Lemma~\ref{lm:char-omega-limit-sets}. Combining now~\eqref{eq:existence-attractor-step-4,1} and~\eqref{eq:existence-attractor-step-4,2} with the continuity assumption~(iii), we obtain
\begin{align*}
\xi = \lim_{k\to\infty} S_{T(t_{n_k}-t)v_{n_k}}\big(t,0,S_{v_{n_k}}(t_{n_k}-t,0,x_{n_k}) \big) = S_v(t,0,x) \in S_{\mathcal{V}}(t,0,\Theta_{\mathcal{V}}),
\end{align*}
which proves the claimed negative invariance of $\Theta_{\mathcal{V}}$ under $S_{\mathcal{V}}$. 
\end{proof}

We note in passing that 
-- just like in the case of (single-valued) semigroups (Remark~1.5 in~\cite{Temam}) -- the asymptotic compactness of a set-valued semiprocess $S_{\mathcal{V}}$ is equivalent to the asymptotic compactness of $S_{\mathcal{V}}$ in the following set-theoretic sense: for every bounded subset $B \subset X$ there exists a compact set $K_B \subset X$ such that 
\begin{align*}
\dist(S_{\mathcal{V}}(t,0,B),K_B) \longrightarrow 0 \qquad (t \to \infty).
\end{align*}
Indeed, the implication from asymptotic compactness to set-theoretic 
asymptotic compactness follows by the first two steps of the above proof, and the reverse implication is easy to see.

\subsection{Asymptotic gains and input-to-state practical stability}
\label{sect:wAG+ISpS,def}

We finally come to the central concept of this paper, namely the (weak) asymptotic gain property. In the following, our input space $\mathcal{U}$ will always be
\begin{align} \label{eq:input-spaces}
\mathcal{U} = \mathcal{U}_{1r_0} := S^{\infty}(\R^+_0,U) \cap \ol{B}_{r_0}^{L^{\infty}}(0)
\quad \text{or} \quad
\mathcal{U} = \mathcal{U}_{2r_0} := L^{\infty}(\R^+_0,U) \cap \ol{B}_{r_0}^{L^{\infty}}(0),
\end{align}
where $r_0 \le \infty$ and where
\begin{align} \label{eq:def-S-infty}
S^{\infty}(\R^+_0,U) := \big\{ u \in L^{\infty}(\R^+_0,U): \, &\text{there exist } u_n \in S_c(\R^+_0,U) \text{ such that } \notag \\
&u_n \longrightarrow u \text{ in } L^{\infty}_{\mathrm{loc}}(\R^+_0,U) \big\}. 
\end{align}
In the above relation, $S_c(\R^+_0,U)$ denotes the set of step functions from $\R^+_0$ to $U$ with compact support, that is, the functions $u: \R^+_0 \to U$ for which there exist finitely many points $0=t_0<t_1< \dotsb < t_m < \infty$ such that $u|_{(t_{i-1},t_i)}$ is constant for all $i \in \{1,\dots,m\}$ and $u|_{(t_m,\infty)} = 0$. It is well-known that $S^{\infty}(\R^+_0,U)$ is a strict subset of $L^{\infty}(\R^+_0,U)$. See the remarks after Proposition~3.4.4 of~\cite{Cohn}, for instance. 

\begin{lm} \label{lm:input-spaces}
If $\mathcal{U}$ is as in~\eqref{eq:input-spaces} with $r_0 \le \infty$ and $T$ is as in~\eqref{eq:transl-sgr}, then $T(h)\mathcal{U} \subset \mathcal{U}$ for every $h \in \R^+_0$. Additionally, $S^{\infty}(\R^+_0,U)$ comprises the bounded piecewise continuous functions from $\R^+_0$ to $U$. 
\end{lm}

\begin{proof}
We only prove the addtional statement because the invariance statement is clear. So, let $u:\R^+_0 \to U$ be bounded and piecewise continuous. 
Since for every $t_0 \in (0,\infty)$ the step functions on $[0,t_0]$ are easily seen to be dense in $PC([0,t_0],U)$ w.r.t.~$\norm{\cdot}_{\infty}$, there exists for every $n \in \N$ a step function $u_{|n}: [0,n] \to U$ such that $$\norm{u_{|n}-u}_{[0,n],\infty} \le 1/n.$$ 
Setting now $u_n(t) := u_{|n}(t)$ for $t \in [0,n]$ and $u_n(t) := 0$ for $t \in (n,\infty)$, we see that $u_n \in S_c(\R^+_0,U)$ for every $n \in \N$ and that for every $t_0 \in (0,\infty)$
\begin{align*}
\norm{u_n-u}_{[0,t_0],\infty} \le \norm{u_n-u}_{[0,n],\infty} = \norm{u_{|n}-u}_{[0,n],\infty} \le 1/n \qquad (n \ge t_0).
\end{align*}
Consequently, $u_n \longrightarrow u$ in $L^{\infty}_{\mathrm{loc}}(\R^+_0,U)$ and therefore $u \in S^{\infty}(\R^+_0,U)$, as desired. 
\end{proof}

Suppose $(S_u)_{u\in\mathcal{U}}$ is a semiprocess family on $X$ with input space $\mathcal{U}$ as in~\eqref{eq:input-spaces} and translation $T$ as in~\eqref{eq:transl-sgr}. Suppose further that the undisturbed (set-valued) semiprocess $S_{\{0\}}$ has a global attractor $\Theta \subset X$. 
Then $(S_u)_{u \in \mathcal{U}}$ is said to be of \emph{asymptotic gain w.r.t.~$\Theta$} iff there exists a comparison function $\gamma \in \mathcal{K}$ such that
\begin{align} \label{eq:def-wAG}
\limsup_{t\to \infty} \norm{x(t,x_0,u)}_{\Theta} \le \gamma(\norm{u}_{\infty})
\end{align}
for every $(x_0,u) \in X \times \mathcal{U}$. 
Also, $(S_u)_{u\in\mathcal{U}}$ is called \emph{input-to-state practically stable w.r.t.~$\Theta$} iff there exist comparison functions $\beta \in \mathcal{KL}$, $\gamma \in \mathcal{K}$ and a constant $c \in \R^+_0$ such that
\begin{align} \label{eq:def-ISpS}
\norm{x(t,x_0,u)}_{\Theta} \le \beta(\norm{x_0}_{\Theta},t) + \gamma(\norm{u}_{\infty}) + c \qquad (t \in \R^+_0)
\end{align}
for every $(x_0,u) \in X \times \mathcal{U}$. In case one can choose $c =0$ in~\eqref{eq:def-ISpS}, then $(S_u)_{u\in\mathcal{U}}$ is called just \emph{input-to-state stable w.r.t.~$\Theta$}. In~\eqref{eq:def-wAG} and~\eqref{eq:def-ISpS} we used the alternative  notation from~\eqref{eq:alternative-notation-for-trajectories}, of course.
\smallskip

%
In the special case where $\Theta = \{0\}$ in the definition above, the relations \eqref{eq:def-uniformly-attractive} and~\eqref{eq:def-neg-invar} imply that $0$ is an attractive equilibrium point of the undisturbed system $S_0$. (Indeed, by~\eqref{eq:def-uniformly-attractive} we have $\norm{S_0(t,0,x_0)} = \dist(S_{\{0\}}(t,0,\{x_0\}), \Theta) \longrightarrow 0$ as $t \to \infty$ for every $x_0 \in X$ and by~\eqref{eq:def-neg-invar} we have $$\{0\} = \Theta \subset S_{\{0\}}(t,0,\Theta) = \{S_0(t,0,0)\}$$ and thus $0 = S_0(t,0,0)$ for all $t \in \R_0^+$.) So, in the special case $\Theta = \{0\}$, the asymptotic gain property as defined above 
reduces to the (weak) asymptotic gain property from~\cite{MiWi16a},   
\cite{Sc18-wISS}. (In \cite{Sc18-wISS} the additional qualifier is used in order to more clearly reflect, 
already in the terminology, the logical relations to the uniform asymptotic gain property and 
the notions of weak, strong, and uniform input-to-state stability.)
In~\cite{Mi18-ISpS} 
the uniform variant of the above (weak) asymptotic gain property is studied 
and abstract theoretical characterizations of input-to-state practical stability are given (Theorem~III.1 and Proposition~IV.7).
\smallskip

We close this section with a remark on how the global attractors 
of set-valued semiprocesses $S_{\{0\}}$, as defined and used 
in this paper, are related to the global attractors of nonlinear semigroups $S_0$ as defined in~\cite{Robinson}, \cite{Temam}, for instance. Apart from the negative invariance condition $\Theta \subset S_0(t,\Theta)$, the global attractors of nonlinear semigroups $S_0$ in the sense of~\cite{Robinson}, \cite{Temam} are also required to satisfy the positive invariance condition 
\begin{align}
S_0(t,\Theta) \subset \Theta \qquad (t \in \R^+_0).
\end{align}
%
So, at first glance, it is not clear whether a global attractor of the set-valued semiprocess $S_{\{0\}}$ is also a global attractor of the corresponding nonlinear semigroup $S_0$. At second glance, however, this turns out to be true. Indeed, let $(S_v)_{v \in \mathcal{V}}$ be a semiprocess family with $\mathcal{V} := \{0\}$ and translation~\eqref{eq:transl-sgr}, let $\Theta$ be a global attractor of the set-valued semiprocess $S_{\{0\}}$, and let $t \in \R^+_0$ and $\eps > 0$. It then follows by~\eqref{eq:def-neg-invar} and~\eqref{eq:cocycle-prop} that
\begin{align} \label{eq:semiproc-attractor-is-semigr-attractor-1}
S_0(t,0,\Theta) = S_{\{0\}}(t,0,\Theta) \subset S_{\{0\}}\big( t,0, S_{\{0\}}(s,0,\Theta) \big) = S_{\{0\}}(t+s,0,\Theta)
\end{align}
for every $s$, and by~\eqref{eq:def-uniformly-attractive} it follows that
\begin{align} \label{eq:semiproc-attractor-is-semigr-attractor-2}
S_{\{0\}}(t+s,0,\Theta) \subset B_{\eps}(\Theta)
\end{align}
for every large enough $s$. Since $\eps > 0$ was arbitrary and $\Theta$ is closed, we conclude from~\eqref{eq:semiproc-attractor-is-semigr-attractor-1} and~\eqref{eq:semiproc-attractor-is-semigr-attractor-2} that $\Theta$ is actually also positively invariant 
and, hence, a global attractor also of the nonlinear semigroup $S_0$ in the sense of~\cite{Robinson}, \cite{Temam}, as desired.

\section{Abstract asymptotic gain and input-to-state practical stability  results}

In this section, we establish our general abstract asymptotic gain result for semiprocess families $(S_u)_{u\in\mathcal{U}}$ with input space $\mathcal{U}$ as in~\eqref{eq:input-spaces} and with translation $T$ as in~\eqref{eq:transl-sgr}. 
In order to do so, we will embed the semiprocess $S_u$ for each individual $u \in \mathcal{U}$ into a family $(S_v)_{v\in\mathcal{V}(u)}$ of semiprocesses with a suitably chosen input space $\mathcal{V}(u)$ which contains $u$. We will 
choose, for given $u \in L^{\infty}(\R^+_0,U)$, 
\begin{align} \label{eq:V(u)-def}
\mathcal{V}(u) := \ol{ \{u(\cdot+h): h \in \R^+_0 \} } \subset L^2_{\mathrm{loc}}(\R^+_0,U),
\end{align}
where the closure is w.r.t.~the weak topology $\ul{\mathcal{T}}$ of the locally convex space $L^2_{\mathrm{loc}}(\R^+_0,U)$ with its standard locally convex topology $\ol{\mathcal{T}}$.
Clearly, 
\begin{align}
u \in \mathcal{V}(u) \qquad \text{and} \qquad \mathcal{V}(0) = \{0\}.
\end{align}
\textcolor{black}{
We will then show that the set-valued semiprocess $S_{\mathcal{V}(u)}$ has a unique global attractor $\Theta_{\mathcal{V}(u)}$ for every disturbance $u \in \mathcal{U}$ and that $\Theta_{\mathcal{V}(u)}$ depends upper semicontinuously on $u$. And from this, in turn, we will be able to  conclude our asymptotic gain result in a relatively easy way. 
}
\smallskip

\textcolor{black}{
At first glance, it might seem strange to embed the semiprocesses $S_u$, as described above, into the set-valued semiprocesses $S_{\mathcal{V}(u)}$ -- instead of just working with the individual semiprocesses $S_u$ themselves. After all, the asymptotic gain property~\eqref{eq:def-wAG} is defined solely in terms of the individual solution curves $S_u$. It might therefore seem natural to work with suitably defined attractors $\Theta_u$ of the disturbed semiprocesses $S_u$ and show the upper semicontinuous dependence of these attractors on $u$. We cannot argue like that, however, because there is no attractor concept that guarantees the negative invariance $\Theta_u \subset S_u(t,0,\Theta_u)$ under the disturbed semiprocesses $S_u$. And such an invariance relation is essential  for the desired upper semicontinuity, see~\eqref{eq:upper-semicont-1} below. In essence, this failure of negative invariance 
is due to the fact that the disturbed semiprocesses $S_u$ for $u \ne 0$ only satisfy the cocycle property
\begin{align*}
S_u(t+h,0,x) = S_{u(\cdot+h)}\big(\cdot,0,S_u(h,0,x)\big)
\end{align*}
but not the semigroup property~\cite{Robinson}, \cite{Temam}, which instead of the translate $u(\cdot+h)$ would feature only $u$ itself.
Collecting all the translates for a given $u$ as we do in~\eqref{eq:V(u)-def} and working with the corresponding set-valued semiprocess $S_{\mathcal{V}(u)}$, however, we can circumvent the difficulties described above. In particular, we get the semigroup-like property~\eqref{eq:sgr-property-S_V} and the negative invariance~\eqref{eq:def-neg-invar}.
}
\smallskip

It should be noticed that a net $(v_i)_{i\in I}$ is weakly convergent to $v$ in $L^2_{\mathrm{loc}}(\R^+_0,U)$ if and only if $v_i|_{[0,t_0]} \longrightarrow  v|_{[0,t_0]}$ weakly in $L^2([0,t_0],U)$ for every $t_0 \in (0,\infty)$. (In order to see this, notice that by standard locally convex theory (Theorem~IV.3.1 of~\cite{Conway}, for instance) a linear functional $\ell: L^2_{\mathrm{loc}}(\R^+_0,U) \to \mathbb{R}$ is $\ol{\mathcal{T}}$-continuous iff there exist $t_0 \in (0,\infty)$ and $M \in \R^+_0$ such that 
\begin{align*}
|\ell(v)| \le M \norm{v}_{[0,t_0],2} \qquad (v \in L^2_{\mathrm{loc}}(\R^+_0,U)).
\end{align*}
And this, in turn, is equivalent to $\ell_{t_0} \in (L^2([0,t_0],U))^*$, 
where $\ell_{t_0}$ is defined by $\ell_{t_0}(v_|) := \ell(v_| \,\&_{t_0} \, 0)$ for all $v_{|} \in L^2([0,t_0],U)$ and $v_| \,\&_{t_0} \, 0$ denotes the zero-extension of $v_|$ to the whole of $\R^+_0$.) 
%

\begin{lm} \label{lm:V(u)}
Suppose $U$ is a separable reflexive space, let $u \in L^{\infty}(\R^+_0,U)$ and let $\mathcal{V}(u)$ and $T$ be as in~\eqref{eq:V(u)-def} and~\eqref{eq:transl-sgr} respectively. 
Then $\mathcal{V}(u)$ 
is a metrizable and, in particular, first-countable, 
compact topological space and $T(h) \mathcal{V}(u) \subset \mathcal{V}(u)$ for all $h \in \R^+_0$. Additionally, for all $s \le t$, one has
\begin{align} \label{eq:estimate-for-transl-bounded-norm}
\int_s^t \norm{v(s)}_U^2 \d s \le (t-s) \norm{u}_{\infty}^2 \qquad (v \in \mathcal{V}(u)). 
\end{align}  
\end{lm}

\begin{proof}
Since $u$ is translation-bounded in $L^2_{\mathrm{loc}}(\R^+_0,U)$ (Section~VII.3 of~\cite{ChVi}) by virtue of $$\sup_{t\in\R^+_0} \int_t^{t+1}\norm{u(s)}_U^2 \d s \le \norm{u}_{\infty}^2 < \infty,$$ the compactness and metrizability of $\mathcal{V}(u)$ follow from Proposition~VII.3.3 and the remark after Definition~VII.3.1 of~\cite{ChVi}. In particular, every $v \in \mathcal{V}(u)$ is the weak limit of a sequence (instead of only a net) of the form $(v_n) = (u(\cdot+h_n))$, that is, $v_n|_{[s,t]} \longrightarrow v|_{[s,t]}$ weakly in $L^2([s,t],U)$ for all $s \le t$. And from this, in turn, it follows that 
\begin{align*}
\norm{v}_{[s,t],2} \le \liminf_{n\to\infty} \norm{v_n}_{[s,t],2} \le (t-s)^{1/2} \norm{u}_{\infty}
\end{align*}  
and that $T(h)v \in \mathcal{V}(u)$, as desired. 
\end{proof}

\begin{lm} \label{lm:upper-semicont}
Suppose that $\mathcal{U}$ is as in~\eqref{eq:input-spaces} with $r_0 \le \infty$, $\mathcal{V}(u)$ and $T$ are as in~\eqref{eq:V(u)-def} and~\eqref{eq:transl-sgr} respectively. Suppose further that $(S_v)_{v\in\mathcal{V}(u)}$ is a semiprocess family on $X$ 
for every $u \in \mathcal{U}$ such that the following assumptions are satisfied:
\begin{itemize}
\item[(i)] there exists a 
bounded subset $B_0 \subset X$ which is absorbing for $S_{\mathcal{V}(u)}$ for every $u \in \mathcal{U}$
\item[(ii)] $S_{\mathcal{V}(u)}$ is asymptotically compact for every $u \in \mathcal{U}$
\item[(iii)] $X \times \mathcal{V}(u) \ni (x,v) \mapsto S_v(t_0,0,x)$ is continuous for every $t_0 \in (0,\infty)$ and $u \in \mathcal{U}$
\item[(iv)] $\dist\big(S_{\mathcal{V}(u)}(t_0,0,B_0),S_{\mathcal{V}(0)}(t_0,0,B_0)\big) \longrightarrow 0$ as $u \to 0$ for every $t_0 \in (0,\infty)$.
\end{itemize}
Then $S_{\mathcal{V}(u)}$ has a unique global attractor $\Theta_{\mathcal{V}(u)}$ for every $u \in \mathcal{U}$ and $u \mapsto \Theta_{\mathcal{V}(u)}$ is upper semicontinuous at $0$, that is,
\begin{align} \label{eq:upper-semicont}
\dist\big(\Theta_{\mathcal{V}(u)},\Theta_{\mathcal{V}(0)}\big) \longrightarrow 0 \qquad (u \to 0). 
\end{align}
\end{lm}

\begin{proof}
Since by our assumptions (i) to (iii) and by Lemma~\ref{lm:V(u)} each semiprocess family $S_{v \in \mathcal{V}(u)}$ satisfies the assumptions of Lemma~\ref{lm:existence-attractor}, that lemma yields the existence of a unique global attractor $\Theta_{\mathcal{V}(u)}$ of $S_{\mathcal{V}(u)}$ for every $u \in \mathcal{U}$. 
%
It remains to show the upper semicontinuity of $u \mapsto \Theta_{\mathcal{V}(u)}$ at $0$. So, let $\eps > 0$ and write $\Theta := \Theta_{\mathcal{V}(0)}$. We first observe that for every $u \in \mathcal{U}$ one has for large enough times $\tau_u$
\begin{align} \label{eq:upper-semicont-1}
\Theta_{\mathcal{V}(u)} \subset S_{\mathcal{V}(u)}\big(\tau_u,0, \Theta_{\mathcal{V}(u)}\big)\subset B_0
\end{align}
by \textcolor{black}{the negative invariance of $\Theta_{\mathcal{V}(u)}$ under $S_{\mathcal{V}(u)}$ and} the absorbingness of $B_0$ for $S_{\mathcal{V}(u)}$. 
Since $\Theta$ is a global attractor for $S_{\mathcal{V}(0)}$, 
there exists a $t_0 \in (0,\infty)$ such that
\begin{align} \label{eq:upper-semicont-2}
S_{\mathcal{V}(0)}(t_0,0,B_0) \subset B_{\eps/2}(\Theta).
\end{align}
Since, moreover, $\dist\big(S_{\mathcal{V}(u)}(t_0,0,B_0),S_{\mathcal{V}(0)}(t_0,0,B_0)\big) \longrightarrow 0$ as $u \to 0$ by our assumption~(iv), there exists a $\delta > 0$ such that
\begin{align} \label{eq:upper-semicont-3}
S_{\mathcal{V}(u)}(t_0,0,B_0) \subset B_{\eps/2}\big( S_{\mathcal{V}(0)}(t_0,0,B_0) \big) \qquad (u \in B_{\delta}^{L^{\infty}}(0)) 
\end{align}
Combining now~\eqref{eq:upper-semicont-1}, \eqref{eq:upper-semicont-2}, \eqref{eq:upper-semicont-3}, we see that
\begin{align}
\Theta_{\mathcal{V}(u)}  \subset S_{\mathcal{V}(u)}\big(t_0,0,\Theta_{\mathcal{V}(u)}\big) \subset S_{\mathcal{V}(u)}\big(t_0,0,B_0\big) \subset B_{\eps}(\Theta)
\end{align}
for all $u \in \mathcal{U}$ with $\norm{u}_{\infty} < \delta$, as desired. 
\end{proof}

With these lemmas at hand, we can now prove our general asymptotic gain result.

\begin{thm} \label{thm:wAG}
Suppose that $\mathcal{U}$ is as in~\eqref{eq:input-spaces} with $r_0 \le \infty$, $\mathcal{V}(u)$ and $T$ are as in~\eqref{eq:V(u)-def} and~\eqref{eq:transl-sgr} respectively. Suppose further that $(S_v)_{v\in\mathcal{V}(u)}$ is a semiprocess family on $X$ 
for every $u \in \mathcal{U}$ such that the assumptions (i) to (iv) from the previous lemma are satisfied. 
Then $S_{\{0\}}$ has a unique global attractor $\Theta$ and $(S_u)_{u\in \mathcal{U}}$ is of asymptotic gain w.r.t.~$\Theta$. 
\end{thm}

\begin{proof}
We use the alternative notation~\eqref{eq:alternative-notation-for-trajectories} and begin by observing that
\begin{align*}
\norm{x(t,x_0,u)}_{\Theta} = \inf \big\{ \norm{x(t,x_0,u)-\theta}: \theta \in \Theta \big\}
\le 
\norm{x(t,x_0,u)-\theta_u} + \dist(\Theta_{\mathcal{V}(u)},\Theta)
\end{align*} 
for every $\theta_u \in \Theta_{\mathcal{V}(u)}$ by the triangle inequality. 
So, taking the infimum over $\theta_u \in \Theta_{\mathcal{V}(u)}$ and using $x(t,x_0,u) \in S_{\mathcal{V}(u)}(t,0,x_0)$, we see that
\begin{align} \label{eq:wAG-thm-1}
\norm{x(t,x_0,u)}_{\Theta} \le \dist(S_{\mathcal{V}(u)}(t,0,x_0),\Theta_{\mathcal{V}(u)}) + \dist(\Theta_{\mathcal{V}(u)},\Theta)
\end{align}
for every $(x_0,u) \in X \times \mathcal{U}$ and $t \in \R^+_0$. 
Since $\Theta_{\mathcal{V}(u)}$ is a global attractor for $S_{\mathcal{V}(u)}$ (Lemma~\ref{lm:upper-semicont}), we conclude 
\begin{align} \label{eq:wAG-thm-2}
\limsup_{t\to\infty} \dist(S_{\mathcal{V}(u)}(t,0,x_0),\Theta_{\mathcal{V}(u)}) = 0 \qquad ((x_0,u) \in X \times \mathcal{U}).
\end{align}
Since, moreover, $u \mapsto \Theta_{\mathcal{V}(u)}$ is upper semicontinuous at $0$ (Lemma~\ref{lm:upper-semicont}), we further conclude that there exists a $\gamma \in \mathcal{K}$ such that
\begin{align} \label{eq:wAG-thm-3}
\dist(\Theta_{\mathcal{V}(u)},\Theta) \le \gamma(\norm{u}_{\infty}) \qquad (u \in \mathcal{U}).
\end{align}
Indeed, let $\delta(u) := \dist(\Theta_{\mathcal{V}(u)},\Theta)$ and $\gamma_0(r) := \sup_{\norm{v}_{\infty} \le r} \delta(v)$,
then by the definition of $\gamma_0$ we have that $\delta(u) \le \gamma_0(\norm{u}_{\infty})$ for all $u \in \mathcal{U}$ and that $\gamma_0$ is monotonically increasing and by the upper semicontinuity~\eqref{eq:upper-semicont} we have that $\gamma_0(r) \longrightarrow 0$ as $r \searrow 0$. And from these three facts about $\gamma_0$ in turn it follows that there exists a $\gamma \in \mathcal{K}$ with $\gamma_0 \le \gamma$ (Lemma~2.5 of~\cite{ClLeSt98}), which proves~\eqref{eq:wAG-thm-3}.
Combining now~\eqref{eq:wAG-thm-1}, \eqref{eq:wAG-thm-2}, \eqref{eq:wAG-thm-3}, we immediately obtain the claimed asymptotic gain property. 
\end{proof}

In order to verify the quite abstract assumptions of the above theorem in our applications, we will use the following corollary. 
Assumption~(i) of this corollary is a dissipation estimate and assumption~(ii) is a compactness condition. We point out that we have to require $r_0 < \infty$ in 
this corollary.

\begin{cor} \label{cor:wAG}
Suppose $\mathcal{U}$ is as in~\eqref{eq:input-spaces} with $r_0 < \infty$ and with $U$ separable and reflexive and let $\mathcal{V}(u)$ and $T$ be as in~\eqref{eq:V(u)-def} and~\eqref{eq:transl-sgr} respectively. Suppose further that $(S_v)_{v\in\mathcal{V}(u)}$ is a semiprocess family on a reflexive space $X$ 
for every $u \in \mathcal{U}$ such that the following assumptions are satisfied:
\begin{itemize}
\item[(i)] there exist a constant $\omega_0 \in (0,\infty)$ and continuous monotonically increasing functions $\sigma, \gamma: \R^+_0 \to \R^+_0$ such that
\begin{align} \label{eq:dissip-estimate,wAG-cor}
\norm{S_v(t,0,x_0)} \le \e^{-\omega_0 t} \sigma(\norm{x_0}) + \gamma(\norm{u}_{\infty}) \qquad (t \in [0,\infty))
\end{align}
for all $(x_0,v) \in  X \times \mathcal{V}(u)$ and all $u \in \mathcal{U}$
\item[(ii)] whenever $x_n \longrightarrow x$ weakly in $X$ and $v_n \longrightarrow v$ weakly in $L^2_{\mathrm{loc}}(\R^+_0,U)$ for some $x_n,x \in X$ and $v_n \in \mathcal{V}(u_n), v \in \mathcal{V}(u)$ and $u_n,u \in \mathcal{U}$, one has the strong convergence 
\begin{align}
S_{v_n}(t_0,0,x_n) \longrightarrow S_v(t_0,0,x) \qquad (n\to \infty)
\end{align}
for every $t_0 \in (0,\infty)$.
\end{itemize}
Then $S_{\{0\}}$ has a unique global attractor $\Theta$ and $(S_u)_{u\in \mathcal{U}}$ is of asymptotic gain w.r.t.~$\Theta$. 
\end{cor}

\begin{proof}
We verify the assumptions~(i) to (iv) of the previous theorem in four steps. 
%
As a first step, we show that the bounded ball ($r_0 < \infty$!)
\begin{align}
B_0 := \ol{B}_{1+\gamma(r_0)}(0) \subset X
\end{align}
is absorbing for $S_{\mathcal{V}(u)}$ for every $u \in \mathcal{U}$. 
Indeed, for every bounded subset $B \subset X$, there is a radius $R$ with $B \subset \ol{B}_R(0)$ and a time $\tau_B$ such that $\e^{-\omega_0 \tau_B} \sigma(R) \le 1$. 
So, by our assumption~(i), 
\begin{align}
\norm{S_v(t,0,x_0)} \le \e^{-\omega_0 t} \sigma(\norm{x_0}) + \gamma(\norm{u}_{\infty}) \le 1 + \gamma(r_0) \qquad (t \ge \tau_B)
\end{align}
for all $(x_0,v) \in B \times \mathcal{V}(u)$ and all $u \in \mathcal{U}$, that is, $S_{\mathcal{V}(u)}(t,0,B) \subset B_0$ for every $t \ge \tau_B$ and every $u \in \mathcal{U}$, as desired. 
\smallskip

As a second step, we show that $S_{\mathcal{V}(u)}$ is asymptotically compact for every $u \in \mathcal{U}$. 
So, let $u \in \mathcal{U}$.
%
As a first preliminary, we observe that the set 
\begin{align}
K_u := S_{\mathcal{V}(u)}(1,0,B_0)
\end{align}
is compact. Indeed, let $(x_n)$ and $(v_n)$ be sequences in $B_0$ and $\mathcal{V}(u)$ respectively. Since $B_0$ as a closed bounded ball is weakly sequentially compact by the assumed reflexivity of $X$ and since $\mathcal{V}(u)$ is sequentially compact in the weak topology of $L^2_{\mathrm{loc}}(\R^+_0,U)$ by Lemma~\ref{lm:V(u)}, 
there exist subsequences and $x \in B_0$ and $v \in \mathcal{V}(u)$ such that $x_{n_k} \longrightarrow x$ weakly in $X$ and $v_{n_k} \longrightarrow v$ weakly in $L^2_{\mathrm{loc}}(\R^+_0,U)$. So, by our assumption~(ii), 
\begin{align}
S_{v_{n_k}}(1,0,x_{n_k}) \longrightarrow S_v(1,0,x) \in K_u \qquad (k\to\infty),
\end{align}
proving the claimed compactness of $K_u$. 
%
As a second preliminary, we observe that the set $K_u$ is absorbing for $S_{\mathcal{V}(u)}$. 
Indeed, by the first step, for every bounded subset $B \subset X$ there exists a time $\tau_B$ such that $S_{\mathcal{V}(u)}(t,0,B) \subset B_0$ for all $t \ge \tau_B$ and therefore we have for every $t \ge \tau_B + 1$ that
\begin{align}
S_{\mathcal{V}(u)}(t,0,B) \subset S_{T(t-1)\mathcal{V}(u)}\big(1,0, S_{\mathcal{V}(u)}(t-1,0,B) \big) \subset S_{\mathcal{V}(u)}(1,0,B_0) = K_u,
\end{align}
proving the claimed absorbingness of $K_u$.
Combining now the compactness and the absorbingness of $K_u$, we immediately 
get the desired asymptotic compactness of $S_{\mathcal{V}(u)}$. 
\smallskip

As a third step, we observe that $X \times \mathcal{V}(u) \ni (x,v) \mapsto S_v(t_0,0,x)$ is continuous for every $t_0 \in (0,\infty)$ and $u \in \mathcal{U}$. Indeed, the sequential continuity of these maps is immediate from our assumption~(ii) and thus the desired continuity immediately follows by the first-countability of $X \times \mathcal{V}(u)$ (Lemma~\ref{lm:V(u)}).
\smallskip

As a fourth and last step, we show that $\dist\big(S_{\mathcal{V}(u)}(t_0,0,B_0),S_{\mathcal{V}(0)}(t_0,0,B_0)\big) \longrightarrow 0$ as $u \to 0$ for every $t_0 \in (0,\infty)$. 
Assuming the contrary, we find a $t_0 \in (0,\infty)$, an $\eps > 0$ and a sequence $(u_n)$ in $\mathcal{U}$ such that
\begin{align*}
\norm{u_n}_{\infty} \le 1/n
\qquad \text{but} \qquad
S_{\mathcal{V}(u_n)}(t_0,0,B_0) \not\subset B_{\eps}\big( S_{\mathcal{V}(0)}(t_0,0,B_0)\big)
\end{align*}
for all $n \in \N$. So, there exist $x_n \in B_0$ and $v_n \in \mathcal{V}(u_n)$ such that
\begin{align} \label{eq:wAG-cor-step-4,1}
S_{v_n}(t_0,0,x_n) \notin B_{\eps}\big( S_{\{0\}}(t_0,0,B_0)\big) \qquad (n \in \N).  
\end{align}
Since $B_0$ 
is weakly sequentially compact by the reflexivity of $X$ and since $$\norm{v_n}_{[0,t],2} \le t^{1/2} \norm{u_n}_{\infty} \le t^{1/2}/n$$ by Lemma~\ref{lm:V(u)},
 there exist a subsequence and an $x \in B_0$ such that $x_{n_k} \longrightarrow x$ weakly in $X$ and $v_{n_k} \longrightarrow 0$ (weakly) in $L^2_{\mathrm{loc}}(\R^+_0,U)$. So, by our assumption~(ii),
\begin{align}
S_{v_{n_k}}(t_0,0,x_{n_k}) \longrightarrow S_0(t_0,0,x) \in S_{\{0\}}(t_0,0,B_0) \qquad (k\to\infty).
\end{align}
Contradiction to~\eqref{eq:wAG-cor-step-4,1}!
\end{proof}

In the special situation of Corollary~\ref{cor:wAG} where the dissipation estimate~\eqref{eq:dissip-estimate,wAG-cor} holds with $\sigma(0) = 0 = \gamma(0)$ (or, in other words, with $\sigma, \gamma \in \mathcal{K}$), that very estimate implies that 
the global attractor of $S_{\{0\}}$ is $\Theta = \{0\}$ and that $(S_u)_{u\in\mathcal{U}}$  is even input-to-state stable w.r.t.~$\Theta$. In the general situation of Corollary~\ref{cor:wAG}, we still get at least input-to-state practical stability. 
In fact, we have the following proposition, in which $r_0 = \infty$ is allowed again.

\begin{prop} \label{prop:ISpS}
Suppose $\mathcal{U}$ is as in~\eqref{eq:input-spaces} with $r_0 \le \infty$ 
and $(S_u)_{u\in\mathcal{U}}$ is a semiprocess family on  $X$ 
such that the following assumptions are satisfied:
\begin{itemize}
\item[(i)] there exist a constant $\omega_0 \in (0,\infty)$ and continuous monotonically increasing functions $\sigma, \gamma: \R^+_0 \to \R^+_0$ such that
\begin{align} \label{eq:ISpS-assumption-1}
\norm{S_u(t,0,x_0)} \le \e^{-\omega_0 t} \sigma(\norm{x_0}) + \gamma(\norm{u}_{\infty}) \qquad (t \in [0,\infty))
\end{align}
for all $(x_0,u) \in  X \times \mathcal{U}$ 
\item[(ii)] $S_{\{0\}}$ has a unique global attractor $\Theta$.
\end{itemize}
Then $(S_u)_{u\in\mathcal{U}}$  is input-to-state practically stable w.r.t.~$\Theta$.
\end{prop}

\begin{proof}
Since $\norm{x_0} \le \inf_{\theta \in \Theta} ( \norm{x_0-\theta} + \norm{\theta} ) \le \norm{x_0}_{\Theta} + \norm{\Theta}$ 
for every $x_0 \in X$ with $\norm{\Theta} := \sup_{\theta \in \Theta} \norm{\theta}$, 
we conclude from~\eqref{eq:ISpS-assumption-1} that
\begin{align} \label{eq:ISpS,1}
\norm{S_u(t,0,x_0)}_{\Theta} 
&= \inf_{\theta\in\Theta} \norm{S_u(t,0,x_0)-\theta} 
\le \e^{-\omega_0 t} \sigma\big( \norm{x_0}_{\Theta} + \norm{\Theta} \big) + \gamma(\norm{u}_{\infty}) + \inf_{\theta\in\Theta} \norm{\theta} \notag\\
&\le \e^{-\omega_0 t} \sigma(2\norm{x_0}_{\Theta}) + \sigma(2 \norm{\Theta}) + \gamma(\norm{u}_{\infty}) + \inf_{\theta\in\Theta} \norm{\theta}
\end{align}
for every $t \in \R^+_0$ and every $(x_0,u) \in X \times \mathcal{U}$. Since, moreover, $r \mapsto \sigma(2r)-\sigma(0)$ and $r \mapsto \gamma(r)-\gamma(0)$ are continuous and monotonically increasing and zero at zero, there exist $\ol{\sigma}, \ol{\gamma} \in \mathcal{K}$ such that
\begin{align} \label{eq:ISpS,2}
\sigma(2r)-\sigma(0) \le \ol{\sigma}(r) \qquad \text{and} \qquad \gamma(r)-\gamma(0) \le \ol{\gamma}(r)
\end{align} 
for all $r \in \R^+_0$. Combining now~\eqref{eq:ISpS,1} and~\eqref{eq:ISpS,2}, we obtain
\begin{align}
\norm{S_u(t,0,x_0)}_{\Theta}  \le \e^{-\omega_0 t} \, \ol{\sigma}(\norm{x_0}_{\Theta}) + \ol{\gamma}(\norm{u}_{\infty}) + c \qquad (t \in \R^+_0)
\end{align}
for every $(x_0,u) \in X\times \mathcal{U}$, where $c := \sigma(0) + \sigma(2 \norm{\Theta}) + \gamma(0) + \inf_{\theta\in\Theta} \norm{\theta}$. And therefore, $(S_u)_{u\in\mathcal{U}}$ is input-to-state practically stable w.r.t.~the global attractor $\Theta$ of $S_{\{0\}}$. 
\end{proof}

\section{Applications to semilinear systems}

In this section, we apply our general asymptotic gain result along with the input-to-state practical stability result from the previous section to semilinear evolution equations. 
We will consider equations of the special form
\begin{align} \label{eq:semilin-evol-eq-u}
\dot{x}(t) = Ax(t) + g(x(t)) + h u(t),
\end{align}
where $A$ is a linear semigroup generator on $X$, $g$ a nonlinear function in $X$, and $h$ is a bounded linear operator from $U$ to $X$. 
In particular, we will consider disturbed nonlinear reaction-diffusion equations of the form
\begin{align} \label{eq:react-diffus-eq}
\begin{split}
\partial_t y(t,\zeta) &= \Delta y(t,\zeta) + g(y(t,\zeta)) + h(\zeta) u(t) \qquad (\zeta \in \Omega) \\
y(t,\zeta) &= 0 \qquad (\zeta \in \partial \Omega)
\end{split} 
\end{align}
on a bounded domain $\Omega \subset \R^d$, where $h \in X := L^2(\Omega,\R)$. As usual, we will embed the equations~\eqref{eq:semilin-evol-eq-u} and~\eqref{eq:react-diffus-eq} into a family 
of equations parametrized by $v \in L^2_{\mathrm{loc}}(\R^+_0,U)$ and, in the case of~\eqref{eq:react-diffus-eq}, we record this family of equations for later reference:
\begin{align} \label{eq:ibvp-v}
\begin{split}
\partial_t y(t,\zeta) &= \Delta y(t,\zeta) + g(y(t,\zeta)) + h(\zeta) v(t) \qquad ((t,\zeta) \in [s,\infty) \times \Omega) \\
y(t,\cdot)|_{\partial \Omega} &= 0 \qquad \text{and} \qquad y(s,\cdot) = y_s \qquad (t \in [s,\infty)).
\end{split} 
\end{align}
\textcolor{black}{Since we will deal with mild and weak solutions (instead of classical solutions), the asymptotic gain and input-to-state practical stability estimates in Corollary~\ref{cor:wAG+ISpS,mild} and~\ref{cor:wAG+ISpS,weak} are valid for all initial states $y_0 \in X = L^2(\Omega)$ (instead of just those from $y_0 \in H^1_0(\Omega)$, say) -- just like our definitions~\eqref{eq:def-wAG} and~\eqref{eq:def-ISpS} require. }

\subsection{Applications in the case of mild solvability}

In this section, we establish an asymptotic gain and an input-to-state practical stability result for the general equation~\eqref{eq:semilin-evol-eq-u} and for the reaction-diffusion equation~\eqref{eq:react-diffus-eq}, taking a mild-solution approach~\cite{Pazy} and taking 
\begin{align}
\mathcal{U} := \mathcal{U}_{1 r_0} = S^{\infty}(\R^+_0,U) \cap \ol{B}_{r_0}^{L^{\infty}}(0).
\end{align}
Accordingly, in this section, 
we have to require the nonlinearity $g$ to be locally Lipschitz continuous.

\begin{lm} \label{lm:max-mild-sol}
Suppose $A$ is a semigroup generator on $X$ and $g: X \to X$ is Lipschitz on bounded subsets of $X$ and $h \in L(U,X)$. Then for every $(s,x_s,v) \in \R^+_0 \times X \times L^2_{\mathrm{loc}}(\R^+_0,U)$ the initial value problem
\begin{align} \label{eq:semilin-ivp-v}
\dot{x}(t) = Ax(t) + g(x(t)) + h v(t) \qquad \text{and} \qquad x(s) = x_s 
\end{align}
has a unique maximal mild solution $x(\cdot,s,x_s,v) \in C([s,T_{s,x_s,v}),X)$. Additionally, if this maximal mild solution is bounded,
\begin{align} \label{eq:max-mild-sol-bd}
\sup_{t \in [s,T_{s,x_s,v})} \norm{x(t,s,x_s,v)} < \infty,
\end{align}
then $x(\cdot,s,x_s,v)$ exists globally in time, that is, $T_{s,x_s,v} = \infty$. 
\end{lm}

\begin{proof}
We can argue in a standard way as in~\cite{Pazy} (Theorem~6.1.4). It should be noticed, however, that the mentioned result covers only the special case where $v \in C(\R^+_0,U)$, 
which is why we sketch the proof for the general case here. 
In order to get the unique maximal mild solvability assertion, we can argue in exactly the same way as in Theorem~6.1.4 of~\cite{Pazy}. 
In order to get the global existence assertion under the additional boundedness assumption~\eqref{eq:max-mild-sol-bd}, we argue by contradiction. So, assume that for some $(s,x_s,v) \in \R^+_0 \times X \times L^2_{\mathrm{loc}}(\R^+_0,U)$ 
\begin{align} \label{eq:max-mild-sol,contradiction-ass}
\sup_{t \in [s,T)} \norm{x(t)} < \infty 
\qquad \text{but} \qquad
T < \infty,
\end{align}
where $T := T_{s,x_s,v}$ and $x := x(\cdot,s,x_s,v) \in C([0,T),X)$ for brevity. 
It is then easy to see that $x$ is uniformly continuous on $[0,T)$ -- just use that under assumption~\eqref{eq:max-mild-sol,contradiction-ass} for every $t_0 \in [s,T)$, the maps $[s,T] \ni r \mapsto \e^{A r}h$ and $[s,T]\times [s,t_0] \ni (r,r') \mapsto \e^{Ar} g(x(r'))$ are uniformly continuous and that
\begin{align*}
\int_{t_0}^T |v(r)| \d r, \quad \int_{t_0}^T \norm{g(x(r))} \d r \longrightarrow 0 \qquad (t_0 \nearrow T).
\end{align*}
Consequently, $x$ extends to a continuous function $\tilde{x} \in C([0,T],X)$ 
and
\begin{align*}
\tilde{x}(t) = \e^{A(t-s)}x_s + \int_s^t \e^{A(t-r)}g(\tilde{x}(r)) \d r + \int_s^t \e^{A(t-r)} h v(r) \d r 
\qquad (t \in [0,T]).
\end{align*}
In other words, $\tilde{x}$ is a mild solution of~\eqref{eq:semilin-ivp-v} on $[0,T]$. Contradiction to the maximality of the mild solution $x$!
\end{proof}

In the rest of this section, we will be dealing with compact semigroups~\cite{Pazy}, that is, semigroups $\e^{A\cdot}$ for which $\e^{At}$ is a compact operator on $X$ for every $t \in (0,\infty)$. 
(In~\cite{EnNa}, such semigroups are called immediately compact.) 

\begin{prop} \label{prop:spec-ISpS-mild}
Suppose $A$ is the generator of a compact semigroup on a Hilbert space $X$ and $g: X \to X$ is Lipschitz on bounded subsets of $X$ and $h \in L(U,X)$ with a seperable Hilbert space $U$. Suppose further that
\begin{align} \label{eq:A-omega-dissipative-and-g-C-damping}
\norm{\e^{At}} \le \e^{-\omega t} \qquad (t \in \R^+_0)
\qquad \text{and} \qquad
\scprd{x,g(x)} \le C + \omega' \norm{x}^2 \qquad (x \in X)
\end{align}
for some constants $\omega \in (0,\infty)$ and $C \in \R$, $\omega' \in (-\infty,\omega)$. Then the maximal mild solutions of~\eqref{eq:semilin-evol-eq-u} with $u \in \mathcal{U} := S^{\infty}(\R^+_0,U)$ 
generate a semiprocess family $(S_u)_{u\in \mathcal{U}}$, $S_{\{0\}}$ has a unique global  attractor $\Theta$, and $(S_u)_{u\in\mathcal{U}}$ is input-to-state practically stable w.r.t.~$\Theta$. 
\end{prop}

\begin{proof}
We proceed in five steps in order to eventually verify the assumptions of Proposition~\ref{prop:ISpS}.
As a first step, we show that for every $(x_0,u) \in D(A)\times S_c(\R^+_0,U)$ the maximal mild solution $x(\cdot,x_0,u)$ (Lemma~\ref{lm:max-mild-sol}) is piecewise continuously differentiable and satisfies $x(t,x_0,u) \in D(A)$ for all $t \in [0,T_{x_0,u})$ and 
\begin{align}
\dot{x}(t,x_0,u) = Ax(t,x_0,u) + g(x(t,x_0,u)) + h u(t) \qquad (t \in (0,T_{x_0,u})\setminus N),
\end{align} 
where $N \subset \R^+_0$ denotes the finite 
set of jump points of $u$.
So, let $(x_0,u) \in D(A)\times S_c(\R^+_0,U)$. Also, let $t_1 < \dotsb < t_m$ be the jump points of $u$ in $(0,T_{x_0,u})$ 
and $t_0 := 0$ and let $t_{m+1} \in (t_m, T_{x_0,u})$ be arbitrary. Set $x_{t_i} := x(t_i,x_0,u)$ for $i \in \{0,\dots,m\}$. 
Since $x(\cdot,x_0,u)|_{[t_{i-1},t_i]}$ is a mild solution of the initial value problem
\begin{align*}
\dot{x}(t) = Ax(t) + g(x(t)) + hu(t) \qquad \text{and} \qquad x(t_{i-1}) = x_{t_{i-1}}
\end{align*}
and since $u|_{(t_{i-1},t_i)} \equiv u_i$ for some constant value $u_i \in U$, 
$x(\cdot,x_0,u)|_{[t_{i-1},t_i]}$ is also a mild solution of the initial value problem
\begin{align} \label{eq:spec-ISpS-mild,ivp-constant}
\dot{x}(t) = Ax(t) + g(x(t)) + h u_i \qquad \text{and} \qquad x(t_{i-1}) = x_{t_{i-1}}
\end{align}
and therefore is a restriction of the corresponding maximal mild solution $x(\cdot,t_{i-1},x_{t_{i-1}},u_i)$ (Lemma~\ref{lm:max-mild-sol}). In short,
\begin{align} \label{eq:spec-ISpS-mild,1}
x(\cdot,x_0,u)|_{[t_{i-1},t_i]} = x(\cdot,t_{i-1},x_{t_{i-1}},u_i)|_{[t_{i-1},t_i]}
\qquad (i \in \{1,\dots, m+1\}). 
\end{align}
Since now $x_{t_0} = x_0 \in D(A)$ and 
the mild solution $x(\cdot,s,x_s,v)$ is a classical solution whenever $x_s \in D(A)$ and $v \equiv v_0$ is constant 
(Theorem~6.1.6 of~\cite{Pazy}), we inductively conclude from~\eqref{eq:spec-ISpS-mild,1} that $x(\cdot,t_{i-1},x_{t_{i-1}},u_i)$ is a classical solution of~\eqref{eq:spec-ISpS-mild,ivp-constant} for every $i \in \{1,\dots, m+1\}$. Consequently, the assertions of the first step now follow in view of~\eqref{eq:spec-ISpS-mild,1} and the arbitrariness of $t_{m+1}$ in $(0,T_{x_0,u})$. 
\smallskip

As a second step, we show that the maximal mild solutions $x(\cdot,x_0,u)$ exist globally in time for $(x_0,u) \in D(A) \times S_c(\R^+_0,U)$ and that there exist a constant $\omega_0 \in (0,\infty)$ and continuous monotonically increasing functions $\sigma, \gamma: \R^+_0 \to \R^+_0$ such that
\begin{align} \label{eq:spec-ISpS-mild,step-2}
\norm{x(t,x_0,u)} \le \e^{-\omega_0 t} \sigma(\norm{x_0}) + \gamma(\norm{u}_{[0,t],\infty}) \qquad (t \in [0,\infty)) 
\end{align} 
for all $(x_0,u) \in D(A) \times S_c(\R^+_0,U)$. So, let $(x_0,u) \in D(A) \times S_c(\R^+_0,U)$ and write $x := x(\cdot,x_0,u)$. It then follows by the first step and our assumption~\eqref{eq:A-omega-dissipative-and-g-C-damping} that 
\begin{align}
\ddt \frac{\norm{x(t)}^2}{2} 
&\le - \omega \norm{x(t)}^2 + C + \omega' \norm{x(t)}^2 + \norm{x(t)}\norm{h} \norm{u(t)} \notag \\
&\le - 2 \omega_0 \frac{\norm{x(t)}^2}{2} + C + \frac{1}{\omega_0} \norm{h}^2 \norm{u(t)}^2 
\qquad (t \in (0,T_{x_0,u}) \setminus N),
\end{align}
where $\omega_0 := (\omega-\omega')/2$ and where in the second inequality we used that $\norm{x(t)}\norm{h} \norm{u(t)} \le \omega_0 \norm{x(t)}^2 + \norm{h}^2 \norm{u(t)}^2/\omega_0$. 
Since $t \mapsto \e^{\omega t} \norm{x(t)}^2$ is piecewise continuously  differentiable by the first step, we conclude 
\begin{align} \label{eq:spec-ISpS-mild,step-2,central-estimate}
\e^{2\omega_0 t} \frac{\norm{x(t)}^2}{2} - \frac{\norm{x_0}^2}{2} 
&\le \int_0^t \e^{2\omega_0 \tau} \Big( C + \frac{1}{\omega_0} \norm{h}^2 \norm{u(\tau)}^2 \Big) \d\tau \notag \\
&\le \frac{\e^{2\omega_0 t}}{2\omega_0} \Big( C + \frac{1}{\omega_0} \norm{h}^2 \norm{u}_{[0,t],\infty}^2 \Big)
\qquad (t \in [0,T_{x_0,u})).
\end{align}
We thus immediately obtain the estimate~\eqref{eq:spec-ISpS-mild,step-2} at least for all $t \in [0,T_{x_0,u})$ and $\sigma(r) := r$ and $\gamma(r) := C_{\omega_0,h} (1+r)$ with some constant $C_{\omega_0,h}$. In particular, this  implies~\eqref{eq:max-mild-sol-bd} and hence $T_{x_0,u} = \infty$ (Lemma~\ref{lm:max-mild-sol}) which, in turn, concludes the proof of the second step.  
\smallskip

As a third step, we show that the maximal mild solutions $x(\cdot,x_0,u)$ exist globally in time and that the estimate 
\begin{align} \label{eq:spec-ISpS-mild,step-3}
\norm{x(t,x_0,u)} \le \e^{-\omega_0 t} \sigma(\norm{x_0}) + \gamma(\norm{u}_{[0,t],\infty}) \qquad (t \in [0,\infty)) 
\end{align} 
holds true also for arbitrary $(x_0,u) \in X \times \mathcal{U}$. So, let $(x_0,u) \in X \times \mathcal{U}$ and choose a sequence $(x_{0n},u_n)$ in $D(A) \times S_c(\R^+_0,U)$ such that $x_{0n} \longrightarrow x_0$ in $X$ and $u_n \longrightarrow u$ in $L^{\infty}_{\mathrm{loc}}(\R^+_0,U)$. We then see by the second step that for every $t_0 \in (0,T_{x_0,u})$ 
\begin{align}
\rho_{t_0} := \sup_{n\in\N} \sup_{\tau \in [0,t_0]} \norm{x(\tau,x_{0n},u_n)} + \sup_{\tau \in [0,t_0]} \norm{x(\tau,x_0,u)} < \infty
\end{align}
and therefore, by the variations of constants formula and the Lipschitz continuity of $g$ on the bounded ball $\ol{B}_{\rho_{t_0}}(0)$, 
\begin{align*}
\norm{x(t,x_{0n},u_n)-x(t,x_0,u)} \le \norm{x_{0n}-x_0} &+ \norm{h} t_0 \norm{u_n-u}_{[0,t_0],\infty} \\
&+ \int_0^t L_{\rho_{t_0}} \norm{x(\tau,x_{0n},u_n)-x(\tau,x_0,u)} \d\tau 
\end{align*}
for all $t \in [0,t_0]$ and $n \in \N$. So, by Gr\"onwall's lemma and the arbitrariness of $t_0$ in $(0,T_{x_0,u})$, we conclude that 
\begin{align}
x(t,x_0,u) = \lim_{n\to\infty} x(t,x_{0n},u_n) \qquad (t \in [0,T_{x_0,u}))
\end{align}
and can thus extend the estimate of the second step -- at least for $t \in [0,T_{x_0,u})$ -- from $(x_{0n},u_n) \in D(A) \times S_c(\R^+_0,U)$ to 
$(x_0,u)$. In particular, 
this extended estimate implies~\eqref{eq:max-mild-sol-bd} and hence $T_{x_0,u} = \infty$ (Lemma~\ref{lm:max-mild-sol}) which, in turn, concludes the proof of the third step.  
\smallskip

As a fourth step, we show that whenever $x_{0n} \longrightarrow x_0$ weakly in $X$, then one has the strong convergence
\begin{align} \label{eq:spec-ISpS-mild-step-4}
x(t_0,x_{0n},0) \longrightarrow x(t_0,x_0,0) \qquad (n \to \infty)
\end{align}
\textcolor{black}{in our Hilbert state space $X$} for every $t_0 \in (0,\infty)$. So, let $x_{0n} \longrightarrow x_0$ weakly in $X$ and $t_0 \in (0,\infty)$. Set $$\rho_{t_0} := \sup_{n \in \N}  \sup_{\tau \in [0,t_0]} \norm{x(\tau,x_{0n},0)} + \sup_{\tau \in [0,t_0]} \norm{x(\tau,x_0,0)} < \infty,$$ which is finite by the third step, 
and let $\eps > 0$. Choose then $\delta \in (0,t_0)$ so small that
\begin{align}
2 L_{\rho_{t_0}} \rho_{t_0} \delta \le \eps / \exp(L_{\rho_{t_0}}  t_0).
\end{align}
We then get, by the variations of constants formula and the Lipschitz continuity of $g$ on the bounded ball $\ol{B}_{\rho_{t_0}}(0)$, that
\begin{align*}
&\norm{x(t,x_{0n},0)-x(t,x_0,0)} 
\le \norm{\e^{A\delta} (x_{0n}-x_0)} +  \int_0^t L_{\rho_{t_0}} \norm{x(\tau,x_{0n},0)-x(\tau,x_0,0)} \d\tau \notag \\
&\qquad \le \norm{\e^{A\delta} (x_{0n}-x_0)} + 2 L_{\rho_{t_0}} \rho_{t_0} \delta + \int_{\delta}^t L_{\rho_{t_0}} \norm{x(\tau,x_{0n},0)-x(\tau,x_0,0)} \d\tau 
\end{align*}
for all $t \in [\delta,t_0]$ and $n \in \N$. So, 
applying Gr\"onwall's lemma and then taking the limit superior, we get 
\begin{align}
\limsup_{n \to \infty} \norm{x(t_0,x_{0n},0)-x(t_0,x_0,0)} \le \lim_{n \to \infty} \norm{\e^{A\delta} (x_{0n}-x_0)} \exp(L_{\rho_{t_0}}  t_0) + \eps = \eps,
\end{align}
where in the last equality we used the compactness of $\e^{A\delta}$. Since $\eps > 0$ was arbitrary, this concludes the proof of the fourth step. 
\smallskip

As a fifth and last step, we 
show that the mild solutions of~\eqref{eq:semilin-evol-eq-u} generate a semiprocess family $(S_u)_{u\in\mathcal{U}}$ which satisfies the assumptions of Proposition~\ref{prop:ISpS}. 
Indeed, by the third step, the maximal mild solutions exist globally in time for all $(x_0,u) \in X \times \mathcal{U}$ and thus generate a semiprocess family $(S_u)_{u\in \mathcal{U}}$ via 
\begin{align*}
S_u(t,s,x_s) := x(t-s,x_s,u(\cdot+s)) \qquad (t \in [s,\infty)),
\end{align*}
see the remarks around~\eqref{eq:from-dyn-syst-with-inputs-to-semiproc-family}. 
Additionally, the third step gives that there exist a constant $\omega_0 \in (0,\infty)$ and continuous monotonically increasing functions $\sigma, \gamma: \R^+_0 \to \R^+_0$ such that the dissipation estimate~\eqref{eq:ISpS-assumption-1} is satisfied for all $(x_0,u) \in X \times \mathcal{U}$. In particular, the dissipation assumption~(i) of Corollary~\ref{cor:wAG} is satisfied with the input space $\mathcal{U}$ from that corollary being $\{0\}$.  
And finally, by the fourth step, the compactness assumption~(ii) of Corollary~\ref{cor:wAG} is satisfied with the input space $\mathcal{U}$ from that corollary being $\{0\}$. Consequently, Corollary~\ref{cor:wAG} with $\mathcal{U} := \{0\}$ gives that $S_{\{0\}}$ has a global attractor $\Theta$. 
So, summarizing, we now have that all the assumptions of Proposition~\ref{prop:ISpS} are satisfied which, in turn, yields the desired conclusion. 
\end{proof}

It should be noticed that if the constant $C$ from~\eqref{eq:A-omega-dissipative-and-g-C-damping} is $C=0$, then $\Theta = \{0\}$ and $(S_u)_{u\in\mathcal{U}}$ is even input-to-state stable -- instead of only input-to-state practically stable -- w.r.t.~$\Theta$. Indeed, this immediately follows by~\eqref{eq:spec-ISpS-mild,step-2,central-estimate}. 
\textcolor{black}{In the next theorem, we have to restrict the input space $\mathcal{U}$ by additionally requiring $r_0 < \infty$ (which was not necessary in the previous proposition).}

\begin{thm} \label{thm:spec-wAG-mild}
Suppose the assumptions of the previous proposition are satisfied. Then the maximal mild solutions of~\eqref{eq:semilin-evol-eq-u} with $u \in \mathcal{U} := \mathcal{U}_{1r_0} = S^{\infty}(\R^+_0,U) \cap \ol{B}_{r_0}^{L^{\infty}}(0)$ and $r_0 < \infty$  
generate a semiprocess family $(S_u)_{u\in \mathcal{U}}$, $S_{\{0\}}$ has a unique global  attractor $\Theta$, and $(S_u)_{u\in\mathcal{U}}$ is of asymptotic gain w.r.t.~$\Theta$. 
\end{thm}

\begin{proof}
We show in three steps that the maximal mild solutions of~\eqref{eq:semilin-ivp-v} generate semiprocess families $(S_v)_{v\in\mathcal{V}(u)}$ for $u \in \mathcal{U}$ that satisfy the assumptions of Corollary~\ref{cor:wAG} which, in turn, yields the desired conclusion. As usual, $\mathcal{V}(u)$ and $T$ are as in~\eqref{eq:V(u)-def} and~\eqref{eq:transl-sgr} respectively. 
%

As a first 
step, we show that whenever $t_n \longrightarrow t$ in $\R^+_0$ and $v_n \longrightarrow v$ weakly in $L^2_{\mathrm{loc}}(\R^+_0,U)$, then $y_n(t_n) \longrightarrow 0$ weakly in $X$, where 
\begin{align} \label{eq:def-y_n(s)}
y_n(s) := \int_0^s \e^{A(s-\tau)} h \big(v_n(\tau)-v(\tau) \big) \d\tau 
\qquad (s \in \R^+_0).
\end{align}
So, let $t_n \longrightarrow t$ in $\R^+_0$ and $v_n \longrightarrow v$ weakly in $L^2_{\mathrm{loc}}(\R^+_0,U)$ and let $z \in X$. We then have
\begin{align} \label{eq:spec-wAG-mild,step-1,1}
\scprd{z,y_n(t_n)} = \int_0^{t_n} \scprd{w(t_n-\tau),v_n(\tau)}_U \d\tau - \int_0^{t_n} \scprd{w(t_n-\tau),v(\tau)}_U \d\tau,
\end{align}
where $w(\tau) := h^* (\e^{A\tau})^* z = h^* \e^{A^*\tau} z$ for $\tau \in \R^+_0$. 
Since $w$ is continuous (Proposition~I.5.14 of~\cite{EnNa} or Proposition~2.8.5 of~\cite{TuWe}), the extension 
$\tilde{w}: \R \to U$ defined by $\tilde{w}|_{(-\infty,0)} := w(0)$ and $\tilde{w}|_{[0,\infty)} := w$ is uniformly continuous on compact subintervals of $\R$ 
and therefore we have for both $v_* = v_n$ and $v_* = v$ that
\begin{align} \label{eq:spec-wAG-mild,step-1,2}
&\norm{ \int_0^{t_n} \scprd{w(t_n-\tau),v_*(\tau)}_U \d\tau - \int_0^{t} \scprd{w(t-\tau),v_*(\tau)}_U \d\tau } \\
&\qquad \le 
\norm{ \int_0^t \scprd{\tilde{w}(t_n-\tau)-\tilde{w}(t-\tau),v_*(\tau)}_U \d\tau } + \norm{ \int_t^{t_n}\scprd{\tilde{w}(t_n-\tau),v_*(\tau)}_U \d\tau } \notag \\
&\qquad \le 
\norm{\tilde{w}(t_n-\cdot)-\tilde{w}(t-\cdot)}_{[0,t],\infty} t^{1/2} C_t + \norm{\tilde{w}(t_n-\cdot)}_{[0,t],\infty}  |t_n-t|^{1/2} C_t \longrightarrow 0 \notag
\end{align}
as $n \to \infty$, where $C_t := \sup_{n\in \N} \norm{v_n}_{[0,t],2}$ is finite by the weak convergence of $(v_n)$. 
Also, by the weak convergence of $(v_n)$ to $v$, we have that
\begin{align} \label{eq:spec-wAG-mild,step-1,3}
\int_0^t \scprd{w(t-\tau),v_n(\tau)}_U \d\tau - \int_0^t \scprd{w(t-\tau),v(\tau)}_U \d\tau
\longrightarrow 0 
\end{align}
as $n \to \infty$. Combining now~\eqref{eq:spec-wAG-mild,step-1,1} with~\eqref{eq:spec-wAG-mild,step-1,2} and~\eqref{eq:spec-wAG-mild,step-1,3}, we finally obtain $\scprd{z,y_n(t_n)} \longrightarrow 0$, as desired. 
\smallskip

As a second 
step, we show that whenever $x_{0n} \longrightarrow x_0$ weakly in $X$ and $v_n \longrightarrow v$ weakly in $L^2_{\mathrm{loc}}(\R^+_0,U)$ for some $x_{0n},x_0 \in X$ and $v_n \in \mathcal{V}(u_n), v \in \mathcal{V}(u)$ and $u_n,u \in \mathcal{U}$ with
\begin{align} \label{eq:spec-wAG-mild,step-2,1}
\sup_{n\in \N} \sup_{t \in [0,T_{x_{0n},v_n})} \norm{x(t,x_{0n},v_n)} < \infty,
\end{align}
then one has the strong convergence
\begin{align} \label{eq:spec-wAG-mild,step-2,2}
x(t_0,x_{0n},v_n) \longrightarrow x(t_0,x_0,v) \qquad (n \to \infty)
\end{align}
for every $t_0 \in (0,T_{x_0,v})$. 
So, let $x_n \longrightarrow x$ weakly in $X$ and $v_n \longrightarrow v$ weakly in $L^2_{\mathrm{loc}}(\R^+_0,U)$ for some $x_n,x \in X$ and $v_n \in \mathcal{V}(u_n), v \in \mathcal{V}(u)$ and $u_n,u \in \mathcal{U}$ such that~\eqref{eq:spec-wAG-mild,step-2,1} is satisfied and let $t_0 \in (0,T_{x_0,v})$. In view of~\eqref{eq:spec-wAG-mild,step-2,1}, we have $T_{x_{0n},v_n} = \infty$ for every $n \in \N$ by Lemma~\ref{lm:max-mild-sol} and thus, in particular, $[0,t_0] \subset [0,T_{x_{0n},v_n})$.  
Set $$\rho_{t_0} := \sup_{n \in \N}  \sup_{\tau \in [0,t_0]} \norm{x(\tau,x_{0n},v_n)} + \sup_{\tau \in [0,t_0]} \norm{x(\tau,x_0,v)} < \infty,$$ which is finite by assumption~\eqref{eq:spec-wAG-mild,step-2,1}, 
and let $\eps > 0$. Choose then $\delta \in (0,t_0)$ so small that
\begin{align}
2 \norm{h} C_{t_0} \delta^{1/2} + 2 L_{\rho_{t_0}} \rho_{t_0} \delta \le \eps / \exp(L_{\rho_{t_0}}  t_0),
\end{align}
where $C_{t_0} := \sup_{n\in \N} \norm{v_n}_{[0,t_0],2} < \infty$. 
We then get, by the variations of constants formula and the Lipschitz continuity of $g$ on the bounded ball $\ol{B}_{\rho_{t_0}}(0)$, that
\begin{align*}
&\norm{x(t,x_{0n},v_n)-x(t,x_0,v)} 
\le \norm{\e^{A\delta} (x_{0n}-x_0)} +  \norm{\e^{A\delta} y_n(t-\delta) } + 2 \norm{h} C_{t_0} \delta^{1/2} \\
&\qquad \qquad  + 2 L_{\rho_{t_0}} \rho_{t_0} \delta + \int_{\delta}^t L_{\rho_{t_0}} \norm{x(\tau,x_{0n},v_n)-x(\tau,x_0,v)} \d\tau 
\end{align*}
for all $t \in [\delta,t_0]$ and $n \in \N$, where $y_n$ is defined as in~\eqref{eq:def-y_n(s)}. 
So, 
applying Gr\"onwall's lemma and then taking the limit superior, we get 
\begin{align}
\limsup_{n \to \infty} \norm{x(t_0,x_{0n},v_n)-x(t_0,x_0,v)} 
&\le 
\lim_{n \to \infty} \bigg( \norm{\e^{A\delta} (x_{0n}-x_0)} + \sup_{t \in [\delta,t_0]} \norm{\e^{A\delta} y_n(t-\delta)} \bigg) \cdot \notag \\
&\qquad \cdot \exp(L_{\rho_{t_0}}  t_0) + \eps = \eps,
\end{align}
where in the last equality we used the compactness of $\e^{A\delta}$ and the first step. Since $\eps > 0$ was arbitrary, this concludes the proof of the second step. 
\smallskip

As a third and last step, we show that the maximal mild solutions of~\eqref{eq:semilin-ivp-v} generate a semiprocess family $(S_v)_{v\in\mathcal{V}(u)}$ for every $u \in \mathcal{U}$ which satisfies the assumptions of Corollary~\ref{cor:wAG}.
We already know from the proof of Propostion~\ref{prop:spec-ISpS-mild} (third step) that the maximal mild solutions $x(\cdot,x_0,u)$ exist globally in time and satisfy the dissipation estimate~\eqref{eq:spec-ISpS-mild,step-3} for all $(x_0,u) \in X \times \mathcal{U}$. Suppose now $(x_0,v) \in X \times \mathcal{V}(u)$ with $u \in \mathcal{U}$ and let $(v_n) = (u(\cdot+h_n))$ be a sequence with $v_n \longrightarrow v$ weakly in $L^2_{\mathrm{loc}}(\R^+_0,U)$. Then $v_n \in \mathcal{U}$ and therefore 
by~\eqref{eq:spec-ISpS-mild,step-3}
\begin{align*}
\norm{x(t,x_0,v_n)} \le \e^{-\omega_0 t} \sigma(\norm{x_0}) + \gamma(\norm{v_n}_{\infty}) 
\le \e^{-\omega_0 t} \sigma(\norm{x_0}) + \gamma(\norm{u}_{\infty}) \qquad (t \in [0,\infty)).
\end{align*}
So, by the second step, we have the strong convergence
\begin{align*}
x(t,x_0,v_n) \longrightarrow x(t,x_0,v) \qquad (n \to \infty)
\end{align*}
for every $t \in (0,T_{x_0,v})$. Combining now the last two relations, we see that
\begin{align} \label{eq:spec-wAG-mild-step-3,1}
\norm{x(t,x_0,v)} \le \e^{-\omega_0 t} \sigma(\norm{x_0}) + \gamma(\norm{u}_{\infty}) \qquad (t \in [0,T_{x_0,v}))
\end{align}
for every $(x_0,v) \in X \times \mathcal{V}(u)$ and $u \in \mathcal{U}$. 
In particular, $T_{x_0,v} = \infty$ (Lemma~\ref{lm:max-mild-sol}). 
We can now define 
\begin{align*}
S_v(t,s,x_s) := x(t-s,x_s,v(\cdot+s)) \qquad (t \in [s,\infty))
\end{align*}
for $(s,x_s,v) \in \R^+_0\times X \times \mathcal{V}(u)$ and $u \in \mathcal{U}$. It is easy to see that $(S_v)_{v\in\mathcal{V}(u)}$ is a semiprocess family, see the remarks around~\eqref{eq:from-dyn-syst-with-inputs-to-semiproc-family}. Additionally, by virtue of~\eqref{eq:spec-wAG-mild-step-3,1}, the dissipation assumption~(i) of Corollary~\ref{cor:wAG} is satisfied. And finally, by virtue of the second step combined with~\eqref{eq:spec-wAG-mild-step-3,1}, 
the compactness assumption~(ii) of Corollary~\ref{cor:wAG} is satisfied as well (note that $\norm{u_n}_{\infty} \le r_0$ for all $u_n \in \mathcal{U}_{1r_0} = \mathcal{U}$). 
\end{proof}

We now specialize from~\eqref{eq:semilin-evol-eq-u} to reaction-diffusion equations~\eqref{eq:react-diffus-eq} and, for that purpose, impose the following assumptions on the nonlinearity $g$ and the inhomogeneity $h$.

\begin{cond} \label{cond:g,h,mild}
\begin{itemize}
\item[(i)] $\Omega$ is a bounded domain in $\R^d$ for some $d \in \N$ 
\item[(ii)] $g: X \to X$ is a function of the form
\begin{align*}
g(y) = \alpha(\norm{y})y \quad (y \in X) \qquad \text{and} \qquad h \in X,
\end{align*}
where $X := L^2(\Omega,\R)$ and $\alpha: \R^+_0 \to \R$ is a locally Lipschitz continuous function 
which, eventually, is smaller than the smallest eigenvalue $\lambda$ of the negative Dirichlet Laplacian on $\Omega$, that is, there exists an $r_0 \in (0,\infty)$ such that
\begin{align} \label{eq:alpha-less-lambda}
\sup_{r \in [r_0,\infty)} \alpha(r) < \lambda.
\end{align}
\end{itemize}
\end{cond}

Suppose that Condition~\ref{cond:g,h,mild} is satisfied and let $s \in \R^+_0$, $y_s \in X$ and $v \in L^2_{\mathrm{loc}}(\R^+_0,\R)$. A function $y \in C([s,T),X)$ is called a \emph{maximal mild solution} of the initial boundary value problem~\eqref{eq:ibvp-v} iff it is a maximal mild solution 
of the corresponding abstract initial value problem~\eqref{eq:semilin-ivp-v} with $A: D(A) \subset X \to X$ being the Dirichlet Laplacian on $\Omega$, that is, 
\begin{align} \label{eq:Dirichlet-Laplacian}
D(A) = \{y \in H^1_0(\Omega): \Delta y \in L^2(\Omega) \}
\qquad \text{and} \qquad 
Ay = \Delta y
\end{align}
(Section~3.6 of~\cite{TuWe}) and with $x_s = y_s$, of course. In these relations, $\Delta y$ is to be understood in the distributional sense. Also, if the boundary $\partial \Omega$ is sufficiently smooth, then 
$D(A) = H^1_0(\Omega) \cap H^2(\Omega)$ (Theorem~3.6.2 of~\cite{TuWe}).

\begin{cor} \label{cor:wAG+ISpS,mild}
Suppose that Condition~\ref{cond:g,h,mild} is satisfied. Then
\begin{itemize}
\item[(i)] the maximal mild solutions of~\eqref{eq:react-diffus-eq} generate a semiprocess family $(S_u)_{u\in \mathcal{U}}$ on $X$ with input space $\mathcal{U} := \mathcal{U}_{1 r_0} = S^{\infty}(\R^+_0,\R) \cap \ol{B}_{r_0}^{L^{\infty}}(0)$ and $r_0 < \infty$
\item[(ii)] $S_{\{0\}}$ has a unique global attractor $\Theta$ and $(S_u)_{u\in \mathcal{U}}$ is of asymptotic gain and input-to-state practically stable w.r.t.~$\Theta$. 
\end{itemize}
\end{cor}

\begin{proof}
We show that the assumptions of Proposition~\ref{prop:spec-ISpS-mild} and Theorem~\ref{thm:spec-wAG-mild} are satisfied with $A$ being the Dirichlet Laplacian from~\eqref{eq:Dirichlet-Laplacian} and with
\begin{align} \label{eq:omega-smallest-ev-of-neg-Dirichlet-Laplacian}
\omega := \inf \sigma(-A).
\end{align}

We will proceed in three steps. 
As a first step, we observe that the Dirichlet Laplacian $A$ on $\Omega$ is the generator of a compact semigroup on $X$ and that
\begin{align} \label{eq:heat-sgr-exp-stable}
\norm{\e^{At}} \le \e^{-\omega t} \qquad (t \in \R^+_0),
\end{align}
where $\omega \in (0,\infty)$ is as in~\eqref{eq:omega-smallest-ev-of-neg-Dirichlet-Laplacian}. 
Indeed, it is well-known that $-A$ is self-adjoint and strictly positive (Propostion~3.6.1 of~\cite{TuWe}) and therefore $\omega = \inf \sigma(-A) > 0$. Consequently, $A$ is the generator of an analytic semigroup satisfying~\eqref{eq:heat-sgr-exp-stable} (Example~3.7.5 of~\cite{ArBaHiNe}). Since analytic semigroups are norm-continuous on $(0,\infty)$ and since $A$ has compact resolvent by Propostion~II.4.25 of~\cite{EnNa} and by the compactness of the embedding $H^1_0(\Omega) \subset L^2(\Omega)$, 
it further follows by Theorem~II.4.29 of~\cite{EnNa} that $e^{A \cdot}$ is a compact semigroup, as desired.
\smallskip

As a second step, we show that $g: X \to X$ is Lipschitz continuous on bounded subsets of $X$.
Indeed, let $\rho > 0$ and let $l_{\rho}$ be a Lipschitz constant of $\alpha|_{[0,\rho]}$. We then have, for all $y,z \in \ol{B}_{\rho}(0)$, that
\begin{align*}
\norm{g(y)-g(z)} \le \big| \alpha(\norm{y})-\alpha(\norm{z}) \big| \norm{y} + \alpha(\norm{z})\norm{y-z}
\le \big( l_{\rho} \rho + \norm{\alpha}_{[0,\rho],\infty} \big) \norm{y-z},
\end{align*}
as desired, where $\norm{\alpha}_{[0,\rho],\infty} := \sup_{r\in [0,\rho]} |\alpha(r)| < \infty$.  
\smallskip

As a third and last step, we show that there exist constants $C \in \R$ and $\omega' \in (-\infty,\omega)$ such that 
\begin{align} \label{eq:estimate-for-yg(y)}
%
\scprd{y,g(y)} \le C + \omega' \norm{y}^2 \qquad (y \in X),
\end{align}
where $\omega$ is as in~\eqref{eq:omega-smallest-ev-of-neg-Dirichlet-Laplacian}. 
Indeed, choose an $r_0 \in (0,\infty)$ such that 
\eqref{eq:alpha-less-lambda} is satisfied with $\lambda := \min \sigma_p(-A)$ and define 
\begin{align*}
\omega' := \sup_{r\in[r_0,\infty)} \alpha(r)
\qquad \text{and} \qquad
C := \sup_{r \in [0,r_0]} |\alpha(r)|r^2.
\end{align*}
Since $A$ has compact resolvent by the first step, we have $\sigma(-A) = \sigma_p(-A)$ (Corollary~IV.1.19 of~\cite{EnNa}) and therefore $\omega' \in (-\infty,\omega)$ by~\eqref{eq:alpha-less-lambda} and, of course, we also have $C \in \R$. 
It is now straightforward to conclude the desired estimate~\eqref{eq:estimate-for-yg(y)}.
\end{proof}

\subsection{An application in the case of weak solvability}

In this section, we establish an asymptotic gain and an input-to-state practical stability result for the reaction-diffusion equation~\eqref{eq:react-diffus-eq}, taking a weak-solution approach~\cite{VaKa06} and taking 
\begin{align}
\mathcal{U} := \mathcal{U}_{2 r_0} = L^{\infty}(\R^+_0,\R) \cap \ol{B}_{r_0}^{L^{\infty}}(0).
\end{align}
Accordingly, in contrast to the previous section, 
we do not have to require the nonlinearity $g$ to be locally Lipschitz continuous -- as a function from $X$ to $X$ -- any more. Instead, it is sufficient to impose the following assumptions on the nonlinearity $g$ (as a function from $\R$ to $\R$) and the inhomogeneity $h$. 
Simple examples for functions $g$ satisfying the three inequalities in~\eqref{eq:cond-g} \textcolor{black}{with an appropriate $p \in [2,\infty)$}  are given by the nonlinearity from the Chaffee--Infante equation (Section~11.5 of~\cite{Robinson}), that is,
\begin{align*}
g(r) = -r^3 + r \qquad (r\in \R)
\end{align*} 
or, more generally, by any polynomial of odd degree with negative leading coefficient.
\textcolor{black}{(Indeed, for these kinds of nonlinearities $g$, the inequalities in~\eqref{eq:cond-g} are satisfied with $p := 2m$, where $2m-1$ is the odd degree of the polynomial $g$.)  
}

\begin{cond} \label{cond:g,h,weak}
\begin{itemize}
\item[(i)] $\Omega$ is a bounded domain in $\R^d$ for some $d \in \N$ with smooth boundary $\partial \Omega$ 
\item[(ii)] $g \in C^1(\R,\R)$ and there exist constants $\alpha_1, \alpha_2, \kappa, \lambda \in (0,\infty)$ such that
\begin{align} \label{eq:cond-g}
-\kappa - \alpha_1 |r|^p \le g(r)r \le \kappa - \alpha_2 |r|^p
\qquad \text{and} \qquad
g'(r) \le \lambda
\qquad (r \in \R)
\end{align}
\textcolor{black}{for some $p \in [2,\infty)$} 
and, moreover, $h \in X := L^2(\Omega,\R)$. 
\end{itemize}
\end{cond}

Suppose that Condition~\ref{cond:g,h,weak} is satisfied \textcolor{black}{for some $p \in [2,\infty)$} and let $s \in \R^+_0$ and $(y_s,v) \in X \times L^2_{\mathrm{loc}}(\R^+_0,\R)$. A function $y \in C([s,\infty),X)$ is 
called a \emph{global weak solution of 
\eqref{eq:ibvp-v}} iff $y(s) = y_s$ and for every $T \in (s,\infty)$ one has
\begin{align}
y|_{[s,T]} \in L^2([s,T],H^1_0(\Omega)) \cap L^p([s,T], L^p(\Omega))
\end{align}
and there exists a (then unique) $z \in L^2([s,T],H^1_0(\Omega)^*) + L^q([s,T], L^q(\Omega))$ such that
\begin{align} \label{eq:def-weak-sol}
\int_s^T \big( z(t), \phi(t)\big) \d t &= -\int_s^T \int_{\Omega} \nabla y(t)(\zeta) \cdot \nabla \phi(t)(\zeta) \d \zeta \d t + \int_s^T \int_{\Omega} g\big( y(t)(\zeta) \big) \, \phi(t)(\zeta) \d \zeta \d t  \notag \\
&\quad + \int_s^T \int_{\Omega} h(\zeta)v(t) \, \phi(t)(\zeta)  \d \zeta \d t 
\end{align}
for every $\phi \in L^2([s,T],H^1_0(\Omega)) \cap L^p([s,T], L^p(\Omega))$. 
See~\cite{VaKa06}, \cite{KaVa09} or~\cite{DaKaSc19} and, for more background information, \cite{ChVi96} or~\cite{ChVi}. In this equation, $(\cdot,\cdot \cdot)$ stands for the dual pairing of $H^1_0(\Omega)^* + L^q(\Omega)$ and $H^1_0(\Omega) \cap L^p(\Omega)$, that is, 
\begin{align} \label{eq:dual-pairing}
(z,\phi) = (z_1,\phi)_{H^1_0(\Omega)^*,H^1_0(\Omega)} + (z_2,\phi)_{L^q(\Omega),L^p(\Omega)}
\end{align}
for every $z = z_1+z_2 \in H^1_0(\Omega)^* + L^q(\Omega)$ and $\phi \in H^1_0(\Omega) \cap L^p(\Omega)$, where $(\cdot,\cdot\cdot)_{H^1_0(\Omega)^*,H^1_0(\Omega)}$ and $(\cdot,\cdot\cdot)_{L^q(\Omega),L^p(\Omega)}$ denote the respective dual pairings. See, for instance, \cite{BeLoe} (Theorem~2.7.1) 
for this duality.
We point out that if $y$ is a global weak solution to~\eqref{eq:ibvp-v}, then for every $T \in (s,\infty)$ there is only one $z \in L^2([s,T],H^1_0(\Omega)^*) + L^q([s,T], L^q(\Omega))$ satisfying~\eqref{eq:def-weak-sol}. And this $z$ is called the \emph{weak or generalized derivative} of $y|_{[s,T]}$. It is denoted by $\partial_t y|_{[s,T]}$ or simply by $\partial_t y$ in the following. 
\smallskip

\textcolor{black}{It is straightforward to verify that Condition~\ref{cond:g,h,weak} implies that conditions (2), (3) and~(4) with $M=0$ from~\cite{VaKa06} are satisfied (note that to the function $f$ from~\cite{VaKa06} corresponds $-g$, the negative of our nonlinearity $g$). In particular, the integrals on the right-hand side of~\eqref{eq:def-weak-sol} are all well-defined. Additionally, by the remarks following~(5) in~\cite{VaKa06} and by Remark~1 from~\cite{VaKa06}, the initial boundary value problem~\eqref{eq:ibvp-v}, for every $s \in \R^+_0$ and every $(y_s,v) \in X \times L_{\mathrm{loc}}^2(\R^+_0,\R)$ has a unique global weak solution provided that Condition~\ref{cond:g,h,weak} is satisfied. }

\begin{cor} \label{cor:wAG+ISpS,weak}
Suppose that Condition~\ref{cond:g,h,weak} is satisfied \textcolor{black}{for some $p \in [2,\infty)$}. Then
\begin{itemize}
\item[(i)] the global weak solutions of~\eqref{eq:react-diffus-eq} generate a semiprocess family $(S_u)_{u\in \mathcal{U}}$ on $X$ with input space $\mathcal{U} := \mathcal{U}_{2 r_0} = L^{\infty}(\R^+_0,\R) \cap \ol{B}_{r_0}^{L^{\infty}}(0)$ and $r_0 < \infty$
\item[(ii)] $S_{\{0\}}$ has a unique global attractor $\Theta$ and $(S_u)_{u\in \mathcal{U}}$ is of asymptotic gain and input-to-state practically stable w.r.t.~$\Theta$. 
\end{itemize}
\end{cor}

\begin{proof}
We show in three steps that the assumptions of Corollary~\ref{cor:wAG} -- and hence also those of Proposition~\ref{prop:ISpS} -- are satisfied. (It should be noticed that the  existence of a unique global attractor of the undisturbed system $S_0$ -- albeit a consequence of Corollary~\ref{cor:wAG} -- is well-known from~\cite{Robinson} (Theorem~11.4), for instance.)
\smallskip

As a first step, we show that the global weak solutions of the initial boundary value problems~\eqref{eq:ibvp-v} for every $u \in \mathcal{U}$ generate a semiprocess family $(S_v)_{v\in\mathcal{V}(u)}$.
Indeed, it is easy to conclude from the remarks in Section~2 of~\cite{VaKa06} (up to Remark~1) that for every $s \in \R^+_0$ and every $(y_s,v) \in X \times L^2_{\mathrm{loc}}(\R^+_0,\R)$ the initial boundary value problem~\eqref{eq:ibvp-v} has a unique global weak solution $y(\cdot,s,y_s,v)$. 
It is also easy to conclude from the definition and the uniqueness of global weak solutions that $(S_v)_{v\in\mathcal{V}(u)}$ defined by
\begin{align*}
S_v(t,s,y_s) := y(t,s,y_s,v) \qquad (v \in \mathcal{V}(u))
\end{align*}
for $u \in \mathcal{U}$ is a semiprocess family on $X$. 
\smallskip

As a second step, we show that the semiprocess families $(S_v)_{v\in\mathcal{V}(u)}$ satisfy the dissipativity assumption~(i) from Corollary~\ref{cor:wAG}. 
Indeed, let $\omega \in (0,\infty)$ be the optimal (largest) constant from Poincar\'{e}'s inequality, choose $\omega' \in (0,\omega)$ and set 
\begin{align} \label{eq:wAG+ISpS,weak,omega_0,def}
\omega_0 := \omega-\omega' \in (0,\infty).
\end{align}
Also, let $u \in \mathcal{U}$ and $(y_0,v) \in X \times \mathcal{V}(u)$ be fixed and write $y := y(\cdot,0,y_0,v) = S_v(\cdot,0,y_0)$ for brevity. It then follows from~\cite{VaKa06} (Section~2) that $t \mapsto \e^{2\omega_0 t} \norm{y(t)}^2$ is absolutely continuous (hence differentiable almost everywhere) with
\begin{align} \label{eq:wAG+ISpS,weak,derivative}
\ddt \Big( \e^{2\omega_0 t} \frac{\norm{y(t)}^2}{2} \Big) = \e^{2\omega_0 t} \Big( \big( \partial_t y(t), y(t) \big) + \omega_0 \norm{y(t)}^2 \Big) 
\end{align}
for almost every $t \in \R^+_0$, where $(\cdot,\cdot\cdot)$ is the dual pairing from~\eqref{eq:dual-pairing}. So, integrating~\eqref{eq:wAG+ISpS,weak,derivative} and applying the definition of weak solutions with $\phi := \e^{2\omega_0 \cdot} y|_{[0,T]}$, we see that
\begin{align*}
&\e^{2\omega_0 T} \frac{\norm{y(T)}^2}{2} - \frac{\norm{y_0}^2}{2}
= \int_0^T \e^{2\omega_0 t} \big( \partial_t y(t), y(t) \big) \d t + \omega_0 \int_0^T \e^{2\omega_0 t} \norm{y(t)}^2 \d t \\
&\quad= -\int_0^T  \e^{2\omega_0 t} \int_{\Omega} |\nabla y(t)(\zeta)|^2 \d\zeta \d t + \int_0^T  \e^{2\omega_0 t} \int_{\Omega} g\big(y(t)(\zeta)\big) y(t)(\zeta) \d\zeta \d t \\
&\qquad + \int_0^T  \e^{2\omega_0 t} \int_{\Omega} h(\zeta) v(t) y(t)(\zeta) \d\zeta \d t + \omega_0 \int_0^T \e^{2\omega_0 t} \norm{y(t)}^2 \d t
\end{align*}
for every $T \in (0,\infty)$. Applying Poincar\'{e}'s inequality, the second inequality from~\eqref{eq:cond-g}, and Young's inequality 
in the form $2 \norm{h} |v(t)| \cdot \norm{y(t)} \le \omega' \norm{y(t)}^2 + \norm{h}^2 |v(t)|^2/\omega'$, we further see that
\begin{align} \label{eq:wAG+ISpS,weak,central-estimate}
\e^{2\omega_0 T} \frac{\norm{y(T)}^2}{2} - \frac{\norm{y_0}^2}{2} 
&\le \big( -(\omega - \omega') + \omega_0 \big) \int_0^T \e^{2\omega_0 t} \norm{y(t)}^2 \d t + \int_0^T \e^{2\omega_0 t} \kappa |\Omega| \d t \notag \\
&\qquad + \int_0^T \e^{2\omega_0 t} \frac{\norm{h}^2}{2 \omega'} |v(t)|^2 \d t
\end{align}
for every $T \in (0,\infty)$. So, as the prefactor of the first integral vanishes by~\eqref{eq:wAG+ISpS,weak,omega_0,def} and as $\int_0^T \e^{2\omega_0 t} |v(t)|^2 \d t \le (1-\e^{-2\omega_0})^{-1} \e^{2\omega_0 T} \norm{u}_{\infty}^2$ by~(2.28) from~\cite{ChVi96} and by~\eqref{eq:estimate-for-transl-bounded-norm}, we conclude from~\eqref{eq:wAG+ISpS,weak,central-estimate} that
\begin{align}
\norm{y(T)}^2 
&\le \e^{-2\omega_0 T} \norm{y_0}^2 + \frac{\kappa |\Omega|}{\omega_0} + \frac{\norm{h}^2}{\omega'} (1-\e^{-2\omega_0})^{-1} \norm{u}_{\infty}^2 \notag \\
&\le \big( \e^{-\omega_0 T} \sigma(\norm{y_0}) + \gamma(\norm{u}_{\infty}) \big)^2 
\qquad (T \in [0,\infty)),
\end{align}
where $\sigma(r) := r$ and $\gamma(r) := C_1 r + C_2$ with $C_1^2 := \frac{\norm{h}^2}{\omega'} (1-\e^{-2\omega_0})^{-1}$ and $C_2^2 := \frac{\kappa |\Omega|}{\omega_0}$. 
\smallskip

As a third and last step, we observe 
that the semiprocess families $(S_v)_{v\in\mathcal{V}(u)}$ satisfy the compactness assumption~(ii) from Corollary~\ref{cor:wAG}.
Indeed, this follows in the same way as Lemma~15 of~\cite{VaKa06}. 
\end{proof}

\section*{Acknowledgements}
\textcolor{black}{We would like to thank the anonymous reviewers for their useful remarks that helped us improve the paper.}
S. Dashkovskiy and O. Kapustyan are partially supported by the German Research Foundation (DFG) and the State Fund for Fundamental Research of Ukraine (SFFRU) through the joint German-Ukrainian grant ``Stability and robustness of attractors of nonlinear infinite-dimensional systems with respect to disturbances'' (DA 767/12-1). 

\begin{small}

\end{small}

\end{document}